\documentclass[11pt]{amsart}
\usepackage{amsthm}

\theoremstyle{plain}
\newtheorem{thm}{Theorem}[section]
\newtheorem{theorem}[thm]{Theorem}

\newtheorem{lemma}[thm]{Lemma}

\newtheorem{proposition}[thm]{Proposition}
\theoremstyle{definition}
\newtheorem{remark}[thm]{Remark}

\newtheorem{definition}[thm]{Definition}

\numberwithin{equation}{section}

 \title[Strong submeasures, intersection of currents, dynamical systems]{Strong submeasures and several applications}
 \author{Tuyen Trung Truong}
   \address{Department of Mathematics, University of Oslo, Blindern 0851 Oslo, Norway}
  \email{tuyentt@math.uio.no}
    \date{\today}

\begin{document}
\maketitle
{\centering\footnotesize To my daughter on her birthday occasion\par}

\begin{abstract}
A strong submeasure on a compact metric space X is a sub-linear and bounded operator on the space of continuous functions on X. A strong submeasure is positive if it is non-decreasing. By Hahn-Banach theorem, a positive strong submeasure is the supremum of a non-empty collection of measures whose masses are uniformly bounded from above. 

We give several applications of strong submeasures in various diverse topics, thus illustrate the usefulness of this classical but largely overlooked notion. The applications include: 

- Pullback and pushforward of all measures by meromorphic selfmaps of compact complex varieties. 

- The existence of invariant  positive strong submeasures for meromorphic maps between compact complex varieties, a notion of entropy for such submeasures (which coincide with the classical ones in good cases) and a version of the Variation Principle.  

- Intersection of every positive closed (1,1) currents on compact K\"ahler manifolds. Explicit calculations are given for self-intersection of the current of integration of some curves $C$ in a compact K\"ahler surface where the self-intersection in cohomology is negative. 

All of these points are new and have not been previously given in work by other authors. In addition, we will apply the same ideas to entropy of transcendental maps of $\mathbb{C}$ and $\mathbb{C}^2$. 
\end{abstract}

\section{Introduction} In dynamical systems and ergodic theory, measures play a crucial role. At least since H. Poincar\'e, the first fundamental step for studying the dynamics of a continuous map $f:X\rightarrow X$ of a compact metric space $X$ is to construct invariant probability measures, that is those measures $\mu$ for which $f_*(\mu )=\mu$, and in particular those with measure entropy equal to the topological entropy. A way to construct invariant measures is to start from a positive measure $\mu _0$, and then to consider any cluster points of the Cesaro's average $\frac{1}{n}\sum _{j=0}^n(f_*)^j(\mu _0)$. To this end, a crucial property is that we can pushforward a probability measure by a continuous map and obtain another probability measure, and that this pushforward is linear on the space of measures. There is also the fundamental result \cite{goodwyn, goodman}, called the Variational Principle, which relates measure entropies of invariant measures of a compact metric space and the topological entropy of the map. When we work with compact complex varieties, usually it is impossible to construct dynamically interesting holomorphic maps $f:X\rightarrow X$, and one must be willing to deal with dominant meromorphic maps $f:X\dashrightarrow X$ in order to go forward. We can still define a notion of entropy for meromorphic maps, and when $X$ is K\"ahler \cite{dinh-sibony3} relates this topological notion with geometrical/cohomological information (called dynamical degrees) of the map. However, only for very special maps which ideas from dynamics of diffeomorphisms of compact Riemann manifolds are applicable so far to the study of meromorphic maps. For these special meromorphic maps, one can construct very special invariant measures $\mu$ with no mass on proper analytic sets, which can be pushed forward. (Here, we recall this pushforward for the reader's convenience. Let $\mu$ be a probability measure with no mass on proper analytic subsets of $X$. Then we define $f_*(\mu )$ as the extension by zero of the probability measure $(f|_{X\backslash I(f)})_*(\mu )$, where $I(f)$ is the indeterminacy set of $f$ and hence $f$ is holomorphic on $X\backslash I(f)$.) For the majority of meromorphic maps, however, there is no obvious such special invariant measures, and one faces difficulty in constructing interesting invariant measures by Cesaro's average as above, since the cluster points of these measures (even if the starting measure $\mu _0$ has no mass on proper analytic subsets) are not guaranteed to have no mass on proper analytic subsets. The problem is then that we really do not know how to pushforward a measure with support in the indeterminacy set $I(f)$ of $f$. For example, if $x_0\in I(f)$, then typically its image under the map $f$ will be of positive dimension, and there is no reasonable way to define the pushforward $f_*(\delta _{x_0})$, of the Dirac measure at $x_0$, as a {\bf measure}. Still, we can ask the following questions: Can we define instead the pushforward $f_*(\delta _{x_0})$ as something else more general than a measure? More importantly, can we hope to obtain some analog of the fundamental results mentioned above for all meromorphic maps? 

Besides the construction using Cesaro's average as above, when $X$ is K\"ahler, there is one other approach of constructing invariant measures, very related to special properties of compact K\"ahler manifolds, by intersecting dynamically interesting special positive closed currents (so-called Green's currents).  Intersection of positive closed currents, in particular those of bidegrees $(1,1)$, is an interesting topic itself with many applications in complex geometry and complex dynamics, and has been intensively studied. The methods employed so far by most researchers in this topic are local in nature, and the resulting intersections are supposed to be positive measures. These local methods also usually provide answers which are not compatible with intersection in cohomology, and the latter is a consideration one needs to take into account in order for the definition to be meaningful. However, again here not much is known about what to do if the currents to intersect are too singular. For example, is there any meaning to assign to self-intersection of the current of integration on a line $L\subset \mathbb{P}^2$?  For an even more interesting example, is there any meaning to assign  to self-intersection of the current of integration on a curve $C$ in a compact K\"ahler surface whose self-intersection in cohomology is $\{C\}.\{C\}=-1$? 

 One key question is that for dynamics of dominant meromorphic maps, should we take more consideration of the indeterminacy set of the map as well as of its iterates (this latter point has been so far not very much pursued by the works in the literature)? Likewise, in considering intersection of positive closed currents, should we  look closer at the singular parts of the currents involved, instead of throwing them away (as in the local approaches mentioned above)? 

In this paper, we give several applications of strong submeasures, a classical but largely overlooked notion, to these questions. In addition, the same ideas can be used for dynamics of transcendental functions on $\mathbb{C}$ and $\mathbb{C}^2$. 

\subsection{Strong submeasures} Let $X$ be a compact metric space. Denote by $\varphi \in C^0(X)$ the sup-norm $||\varphi ||_{L^{\infty}}=\sup _{x\in X}|\varphi (x)|$. We recall that by Riesz representation theorem (see \cite{rudin}), on $X$ a measure of bounded mass  $\mu$ is the same as a {\bf linear} operator $\mu: C^0(X)\rightarrow \mathbb{R}$, that is $\mu (\lambda _1\varphi _1+\lambda _2\varphi _2)=\lambda _1\mu (\varphi _1)+\lambda _2(\varphi _2)$ for all $\varphi _1,\varphi _2\in C^0(X)$ and constants $\lambda _1,\lambda _2$; which is {\bf bounded}: $|\mu (\varphi )|\leq C||\varphi ||_{L^{\infty}}$ for a constant $C$ independent of $\varphi \in C^0(X)$, and {\bf positive}: $\mu (\varphi )\geq 0$ whenever $\varphi \in C^0(X)$ is non-negative. Note that we can then choose $C=\mu (1)$ (the mass of the measure $\mu $) and by linearity the positivity is the same as having 
\begin{equation}
\mu (\varphi _1)\geq \mu (\varphi _2)
\label{EquationPositivity}\end{equation} 
for all $\varphi _1,\varphi _2\in C^0(X)$ satisfying $\varphi _1\geq \varphi _2$.  For later reference, we denote by $M(X)$ the set of signed measures on $X$ and by $M^+(X)$ the set of positive measures on $X$.

Inspired by this, we now define strong submeasures and positive strong submeasures. We recall that a functional $\mu :C^0(X)\rightarrow \mathbb{R}$ is {\bf sub-linear} if $\mu (\varphi _1+\varphi _2)\leq \mu (\varphi _1)+\mu (\varphi _2)$ and $\mu (\lambda \varphi )=\lambda \mu (\varphi )$ for $\varphi _1,\varphi _2,\varphi \in C^0(X)$ and a non-negative constant $\lambda$. A {\bf strong submeasure} is then simply a sub-linear functional $\mu :C^0(X)\rightarrow \mathbb{R}$ which is also {\bf bounded}, that is there is a constant $C>0$ so that for all $\varphi \in C^0(X)$ we have $|\mu (\varphi )|\leq C||\varphi ||_{L^{\infty}}$. The least such constant $C$ is called the norm of $\mu$ and is denoted by $||\mu ||$. A strong submeasure $\mu$ is {\bf positive} if it is non-decreasing, that is for all $\varphi _1\geq \varphi _2$ we have $\mu (\varphi _1)\geq \mu (\varphi _2)$. It is easy to check that a strong submeasure is Lipschitz continuous $|\mu (\varphi _1)-\mu (\varphi _2)|\leq ||\mu ||\times ||\varphi _1-\varphi _2||_{L^{\infty}}$, and convex $\mu (t_1\varphi _1+t_2\varphi _2)\leq t_1\mu (\varphi _1)+t_2\mu (\varphi _2)$ for $t_1,t_2\geq 0$. We denote by $SM(X)$ the set of all strong submeasures on $X$, and by $SM^+(X)$ the set of all positive strong submeasures on $X$. Note that $M(X)\subset SM(X)$ and $M^+(X)\subset SM^+(X)$. 

By a simple application of Hahn-Banach's extension theorem (see \cite{rudin2}) and Riesz representation theorem (see \cite{rudin}), we have the following characterisation, whose proof is left out, of strong submeasures and positive strong submeasures. 
\begin{theorem} Let $X$ be a compact metric space, and $\mu :C^0(X)\rightarrow \mathbb{R}$ an operator. 

1) $\mu $ is a strong submeasure if and only if there is a non-empty collection $\mathcal{G}$ of signed measures $\chi =\chi ^+-\chi ^-$ where $\chi ^{\pm}$ are measures on $X$ so that $\sup _{\chi =\chi ^{+}-\chi ^-\in \mathcal{G}}\chi ^{\pm}(1)<\infty$, and: 
\begin{equation}
\mu (\varphi )=\sup _{\chi \in \mathcal{G}}\chi (\varphi ), 
\label{EquationExampleSubmeasure}\end{equation}
for all continuous functions $\varphi$.

2) $\mu$ is a positive strong submeasure if and only if there is a non-empty collection $\mathcal{G}$ of (positive) measures on a compact metric space $X$ so that $\sup _{\chi \in \mathcal{G}}\chi (1)<\infty$, and:
\begin{equation}
\mu (\varphi )=\sup _{\chi \in \mathcal{G}}\chi (\varphi ),
\label{EquationExamplePositiveSubmeasure}\end{equation}
for all continuous functions $\varphi$. 
\label{TheoremHahnBanach}\end{theorem}

The next paragraphs discuss the natural topology on the space of strong submeasures. 

\begin{definition} We say that a sequence $\mu _1,\mu _2,\ldots \in SM(X)$ weakly converges to $\mu \in SM(X)$ if $\sup _n||\mu _n||<\infty$ and 
\begin{equation}
\lim _{n\rightarrow \infty}\mu _n(\varphi )=\mu (\varphi )
\label{EquationWeakConvergence}\end{equation} 
for all $\varphi \in C^0(X)$. We use the notation $\mu _n\rightharpoonup \mu$ to denote that $\mu _n$ weakly converges to $\mu$. 
\end{definition}

If $\mu _1,\mu _2:~C^0(X)\rightarrow \mathbb{R}$, we define $\max \{\mu _1,\mu _2\}:~C^0(X)\rightarrow \mathbb{R}$ by the formula $\max \{\mu _1,\mu _2\}(\varphi )=\max \{\mu _1(\varphi ), \mu _2(\varphi )\}$. Theorem \ref{TheoremSubmeasureBasicProperty} shows that submeasures also have properties similar to measures, such as weak compactness.  

For later use, we recall that  for a compact subset $A\subset X$, we have (\cite{doob}):
\begin{equation}
\mu (A)=\inf _{\phi \in C^0(X,\geq 1_A)}\mu (\phi ). 
\label{EquationMeasureClosedSet}\end{equation}
Here $1_A:X\rightarrow \{0,1\}$ is the characteristic function of $A$, that is $1_A(x)=1$ if $x\in A$ and $=0$ otherwise, $C^0(X)$ is the space of continuous functions from $X$ into $\mathbb{R}$, and for any bounded function $g:X\rightarrow \mathbb{R}$ we use the notations 
\begin{equation}
C^0(X, \geq g)=\{\phi \in C^0(X):~\phi \geq g\}. 
\label{EquationContinuousFunctionsBiggerG}\end{equation}
Moreover, for any {\bf open} set $B\subset X$ we have (\cite{doob})
\begin{equation}
\mu (B)=\sup _{A~\mbox{compact}\subset B}\mu (A). 
\label{EquationMeasureOpenSet}\end{equation}

Like measures, {\bf positive} strong submeasures give rise naturally to {\bf set functions}.  On a compact metric space $X$,  recall that a function $g :X\rightarrow \mathbb{R}$ is upper-semicontinuous if for every $x\in X$
\begin{equation}
\limsup _{y\rightarrow x}g (y)\leq g (x).
\label{EquationUpperSemicontinuity}\end{equation}
For example, the characteristic function of a closed subset $A\subset X$ is upper-semicontinuous. By Baire's theorem \cite{baire}, if $g$ is a bounded upper-semicontinuous function on $X$ then the set $C^0(X,\geq g)$ is non-empty and moreover
\begin{eqnarray*}
g=\inf _{\varphi \in C^0(X,\geq g)}\varphi . 
\end{eqnarray*}
More precisely, there is a sequence of continuous functions $g_n$ on $X$ decreasing to $g$. Hence, if $\mu$ is a {\bf measure}, we have by Lebesgue and Levi's monotone convergence theorem in the integration theory that
\begin{eqnarray*}
\mu (g)=\lim _{n\rightarrow\infty}\mu (g_n)=\inf _{\varphi \in C^0(X,\geq g)}\mu (\varphi ). 
\end{eqnarray*}

Inspired by this and (\ref{EquationMeasureClosedSet}), if $\mu$ is an {\bf arbitrary} strong submeasure, we define for any upper-semicontinuous function $g$ on $X$ the value
\begin{equation}
E(\mu) (g):=\inf _{\varphi \in C^0(X,\geq g)}\mu (\varphi )\in [-\infty ,\infty ).
\label{EquationSubmeasureUpperSemicontinuous}\end{equation}
Then for a closed set $A\subset X$, we define $\mu  (A):=E(\mu )(1_{A})$ where $1_A$ is the characteristic function of $A$. If $\mu$ is positive, we always have $\mu (A)\geq 0$. Then, for an open subset $B\subset X$, following (\ref{EquationMeasureOpenSet}) we define $\mu (B):=\sup \{\mu (A):~A$ compact  $\subset B\}$. Denote by $BUS(X)$ the set of all bounded upper-semicontinuous functions on $X$.  Theorem \ref{TheoremSubmeasureUpperSemicontinuous} proves some basic properties of this operator, similar to those of submeasures. 

If we have a positive strong submeasure $\mu$, and define for any Borel set $A\subset X$ the number $\widetilde{\mu}(A)=\inf \{\mu (B):$ $B$ open, $A\subset B\}$, then we see easily from part 4) of Theorem \ref{TheoremSubmeasureUpperSemicontinuous} that: i) $\widetilde{\mu}(\emptyset )=0$, ii) $\widetilde{\mu}(A_1)\leq \widetilde{\mu}(A_2)$ for all Borel sets $A_1\subset A_2$ and iii) 
$\widetilde{\mu}(A_1\cup A_2)\leq \widetilde{\mu}(A_1)+\widetilde{\mu}(A_2)$. Such $\widetilde{\mu}$ are known in the literature as submeasures (see e.g. \cite{talagrand}), and hence it is justified to call our objects $\mu$ positive strong submeasures.  

\subsection{Pushforward of positive strong submeasures}

In this subsection we discuss several results concerning pushforward of positive strong submeasures by dominant meromorphic maps. 

First, we recall that to pushforward a measure by a continuos map $f:X\rightarrow Y$ is the same as pullback continuous functions by $f$. In fact, the pushforward $f_*(\mu )$ of the measure $\mu$ is a measure which acts on continuous functions $\varphi :X\rightarrow \mathbb{R}$ by $f_*(\mu )(\varphi ):= \mu (f^*\varphi )$, where $f^*\varphi :X\rightarrow \mathbb{R}$ is the composition $\varphi \circ f$ and hence is also continuous. 

Now we consider the pushforward of a measure, and more generally a strong submeasure, by a dominant meromorphic map $f:X\dashrightarrow Y$ between compact complex varieties. We would like, as in the above paragraph, to define the pullback $f^*(\varphi )$ of a continuous function $\varphi :Y\rightarrow \mathbb{R}$. However, in general there is no reasonable way to define $f^*(\varphi )$ as a continuous function, because of the existence of  indeterminacy points. In general, $f^*(\varphi )$ is only continuous on an open dense subset of $X$. Fortunately, there is a canonical way to define the pullback $f^*(\varphi )$ as an upper-semicontinuous function, using only the value of $f^*(\varphi )$ on the open dense set where it is continuous. The key to this is the next result. 
\begin{proposition}
Let $X$ be a compact metric space, $U\subset X$ an open dense set, and $g:~U\rightarrow \mathbb{R}$ a bounded upper semicontinuous function. Define $E(g):X\rightarrow \mathbb{R}$ as follows: If $x\in U$ then $E(g)(x):=g(x)$, and if $x\in X\backslash U$ then 
\begin{eqnarray*}
E(g)(x):=\limsup _{y\in U,~y\rightarrow x}g(y).
\end{eqnarray*} 
Then 

1) $E(g)$ is a bounded upper-semicontinuous function, and $E(g)|_U=g$. In other words, $E(g)$ is a bounded upper-semicontinuous extension of $g$. 

2) If $g$ is continuous on $U$, $U_1\subset U$ is another open dense set of $X$ and $g_1=g|_{U_1}$, then $E(g_1)=E(g)$. 

3) Moreover, $E(g_1+g_2)\leq E(g_1)+E(g_2)$ for any $g_1,g_2:U\rightarrow \mathbb{R}$ bounded upper-semicontinuous functions. 
\label{PropositionUpperSemicontinuousExtension}\end{proposition}
{\bf Remark.} On the other hand, if $g$ is not continuous on $U$ then it is easy to construct examples for which the conclusion of part 2) in the proposition does not hold.  

Let $f:X\dashrightarrow Y$ be now a dominant meromorphic map. Let $\Gamma _f\subset X\times Y$ be the graph of $f:X\dashrightarrow Y$,  $\pi _X,\pi _Y:~X\times Y\rightarrow X,Y$ be the projections, and $\pi _{X,f},\pi _{Y,f}$ the restrictions of these maps to $\Gamma _f$. Then $\Gamma _f$ is also a compact complex variety, and hence a compact metric space under the induced metric from $X\times Y$, and the holomorphic map $\pi _{X,f}:~\Gamma _f\rightarrow X$ is bimeromorphic. Moreover, since the map $f$ is dominant, we have that the map $\pi _{Y,f}:\Gamma _f\rightarrow Y$ is surjective. Hence, we can find an open dense set $U\subset \Gamma _f$ so that $\pi _{X,f}$ is isomorphic from $U$ onto its image in $X$. 

\begin{definition} Using Proposition \ref{PropositionUpperSemicontinuousExtension}, we define $(\pi _{X,f})_*(\phi )$, where $\phi \in C^0(\Gamma _f)$, to be the upper-semicontinuous function  $E((\pi _{U,f})_*(\phi ))$ on $X$. We emphasise that it is {\bf globally defined} on the whole of $X$, and is not changed if we replace $U$ by one open dense subset of it. 

Similarly, if $\dim (X)=\dim (Y)$, then there is $W\subset Y$ a Zariski dense open subset so that $\pi _{W,f}:\pi _{Y,f}^{-1}(W)\subset \Gamma _f\rightarrow W$ is a finite covering map. Hence for any continuous function $\phi$ on $\Gamma _f$ we easily have that $(\pi _{W,f})_*(\phi )$ is upper-semicontinuous on $Y$ and is independent of the choice of $W$ (see Section \ref{SectionPullbackPushforward}). 
\end{definition}

We have finally a canonical definition of pullback of continuous functions by dominant meromorphic functions. 

\begin{definition} Let $\varphi \in C^0(X)$ and $\psi \in C^0(Y)$. We denote by $f^*(\psi )$ the above upper-semicontinuous function $(\pi _{X,f})_*(\pi _{Y,f}^*(\psi ))$. Moreover, when $\dim (X)=\dim (Y)$, we define also by $f_*(\varphi )$ the above upper-semicontinuous function $(\pi _{Y,f})_*(\pi _{X,f}^*(\varphi ))$.
 \end{definition}

Using Proposition \ref{PropositionUpperSemicontinuousExtension} and the above upper-semicontinuous pushforward of functions,  we can finally define following Theorem \ref{TheoremSubmeasureUpperSemicontinuous}, the following pullback operator $\pi _{X,f}^*:SM^+(X)\rightarrow SM^+(\Gamma _f)$:
\begin{equation}
\pi _{X,f}^*(\mu )(\varphi ):=\inf _{\psi \in C^0(X,\geq (\pi _{X,f})_*(\varphi ))}\mu (\psi ). 
\label{EquationSubmeasurePullback}\end{equation}
Then, as in the case of continuous maps, we define $f_*(\mu )$, of a positive strong submeasure $\mu$, by the formula $f_*(\mu )=(\pi _{Y,f})_*\pi _{X,f}^*(\mu )$. 

\begin{definition} 
For convenience, we write here the final formula for pushforwarding a strong submeasure by a dominant meromorphic map $f$:
\begin{equation}
f_*(\mu )(\varphi ):=\inf _{\psi \in C^0(X,\geq (\pi _{X,f})_*(\pi _{Y,f}^*(\varphi )))}\mu (\psi ). 
\label{EquationSubmeasurePushforwardMeromorphic}\end{equation}

Similarly, when $\dim (X)=\dim (Y)$, we can define the pullback of a strong submeasure by the formula: 
\begin{equation}
f^*(\mu )(\varphi ):=\inf _{\psi \in C^0(Y,\geq (\pi _{Y,f})_*(\pi _{X,f}^*(\varphi )))}\mu (\psi ). 
\label{EquationSubmeasurePullbackMeromorphic}\end{equation}
\end{definition}

{\bf Remark.} In the above definitions of pullback and pushforward (of positive strong submeasures) or upper-semicontinuous pullback and pushforward (of continuous functions) by meromorphic maps, if we replace the graph $\Gamma _f$ by its resolutions of singularities, we will obtain the same result. This will be proven in Section \ref{SectionPullbackPushforward}. 

Before going further, let us calculate explicitly one simple but interesting example. 

{\bf Example 1.} Let $\pi :Y\rightarrow X$ be the blowup of $X$ at a point $p$, and $V\subset Y$ the exceptional divisor. Let $\delta _p$ be the Dirac measure at $p$. Then for any continuous function $\varphi $ on $Y$, we have
\begin{eqnarray*}
\pi ^*(\delta _p)(\varphi )=\max _{y\in V}\varphi (y).
\end{eqnarray*}
Therefore, $\pi ^*(\delta _p)$ is {\bf not} a measure. In particular, if $A\subset Y$ is a closed set then $\pi ^*(\delta _p) (A)=\inf _{\varphi \in C^0(X,\geq 1_A)}\pi ^*(\delta _p)(\varphi )$ is $\delta _p(\pi (A\cap Y))$.   

\begin{proof}[Proof of Example 1] By definition 
\begin{eqnarray*}
\pi ^*(\delta _p)(\varphi )=\inf _{\psi \in C^0(Y,\geq \pi _*(\varphi ))}\delta _p(\psi )=\inf _{\psi \in C^0(Y,\geq \pi _*(\varphi ))}\psi (p). 
\end{eqnarray*}
Since $\pi :Y\backslash V\rightarrow X\backslash \{p\}$ is an isomorphism, it is easy to check that $\pi _*(\varphi )(p)=\max _{y\in V}\varphi (y)$. Therefore, for any $\psi \in C^0(Y,\geq \pi _*(\varphi ))$, we have $\psi (p)\geq \max _{y\in V}\varphi (y)$. Hence by definition $\pi ^*(\delta _p)(\varphi )\geq \max _{y\in V}\varphi (y)$. On the other hand, for any $\epsilon >0$, choose a small neighborhood $U_{\epsilon}$ of $p$ so that 
\begin{eqnarray*}
\sup _{y\in \pi ^{-1}(U_{\epsilon })}\varphi (y)\leq \epsilon + \max _{y\in V}\varphi (y). 
\end{eqnarray*}
It follows that $\sup _{U_{\epsilon }}\pi _*(\varphi )\leq \epsilon + \max _{y\in V}\varphi (y)$. Since $\pi _*(\varphi )$ is continuous on $X\backslash \{p\}$, it follows by elementary set theoretic topology that we can find a continuous function $\psi $ on $X$ so that $\psi \geq \pi _*(\varphi )$ and $\sup _{U_{\epsilon}}\psi \leq \epsilon + \max _{y\in V}\varphi (y)$. It follows that $\pi ^*(\delta _p)(\varphi )\leq \epsilon + \max _{y\in V}\varphi (y)$. Since $\epsilon $ is an arbitrary positive number, we conclude from the above discussion that $\pi ^*(\delta _p)(\varphi )=\max _{y\in V}\varphi (y)$. Similarly, we can show that $\pi ^*(\delta _p) (A)=\delta _p(\pi (A\cap Y))$. 
\end{proof}

In Theorem \ref{TheoremSubmeasurePushforwardMeromorphic} we will prove some fundamental properties of these pushforward and pullback operators on positive strong submeasures. We extract here some most interesting properties. 

\begin{theorem} Let $f:X\dashrightarrow Y$ be  a dominant meromorphic map of compact complex varieties. 

i) If $\mu _n\in SM^+(X)$ weakly converges to $\mu$,  and $\nu$ is a cluster point of $f_*(\mu _n)$, then $\nu \leq f_*(\mu )$. If $f$ is holomorphic, then $\lim _{n\rightarrow\infty}f_*(\mu _n)=f_*(\mu )$. 

ii) For any positive strong submeasure $\mu$, we have $f_*(\mu )=\sup _{\chi \in \mathcal{G}(\mu)}f_*(\chi )$, where $\mathcal{G}(\mu )=\{\chi :$ $\chi $ is a measure and $\chi \leq \mu \}$. 

iii) For positive strong submeasures $\mu _1,\mu _2$, we have $f_*(\mu _1+\mu _2)\geq f_*(\mu _1)+f_*(\mu _2)$. 

\label{TheoremSubmeasurePushforwardMeromorphicShortVersion}\end{theorem}

 In Example 2 in Section \ref{SectionPullbackPushforward}, we will show that strict inequality can happen in part i) in general. It also shows that, in contrast to the case of a continuous map - see Section 4, part ii) does not hold in general if we replace $\mathcal{G}(\mu )$ by a smaller set $\mathcal{G}$ (still satisfying $\mu =\sup _{\chi \in \mathcal{G}}\chi$). Generalising Example 1, we will give an explicit expression in Theorem \ref{TheoremPushforwardBlowup} for the pushforward  $f_*$ and a corresponding collection $\mathcal{G}$ in part 2) of Theorem \ref{TheoremHahnBanach}. We also remark that several results in Theorem \ref{TheoremSubmeasurePushforwardMeromorphic} (such as parts 1,2,3) can be extended easily to meromorphic correspondences. 

\subsection{Invariant positive strong submeasures and entropy} In this subsection we apply the previous results to explore invariant strong submeasures and entropy of dominant meromorphic maps $f:X\dashrightarrow X$, where $X$ is a compact complex variety.

We first recall some fundamental results on entropy in continuous dynamics. Consider for the moment $f:X\rightarrow X$ a continuous map of a compact metric space $X$. Then there is a notion of topological entropy $h_{top}(f)$ \cite{adler-konheim-mcandrew}, which is a number between $[0,\infty ]$ and measures the complexity of the map $f$. For example, if $f$ is periodic, then its entropy is zero, and it is regarded dynamically uninteresting. For an $f$-invariant probability measure $\mu$, we also have the notion of measure entropy $h_{\mu}(f)$, which will be recalled in Section 4. We have the famous Variational Principle \cite{goodman, goodwyn}:
\begin{eqnarray*}
h_{top}(f)=\sup _{\mu }h_{\mu}(f),
\end{eqnarray*}
 where the supremum is all over invariant probability measures. Also, we recall that we can construct all invariant measures of $f$ by the following manner: Choose any measure $\mu$ on $X$, then take the Cesaro's averages $\frac{1}{n}\sum _{j=0}^n(f_*)^j(\mu)$, and finally pick cluster points of such sequences. 

The topological and measure entropies have some good behaviours in relation to semi-conjugacies. Let $g:Y\rightarrow Y$ be another continuous map of a compact metric space. Assume that there is a surjective continuous map $\pi :Y\rightarrow X$ so that $\pi \circ g=f\circ \pi$. In this case, we say that $g$ is semi-conjugated to $f$ via $\pi$. Then $h_{top}(g)\geq h_{top}(f)$. Moreover, for every probability measure $\mu$ on $X$ invariant by $f$, there is a probability measure $\hat{\mu}$ on $Y$ invariant by $g$ so that $\pi _*(\hat{\mu})=\mu$. Moreover, in this case we have $h_{\hat{\mu}}(g)\geq h_{\mu}(f)$.

Now we turn to dynamics of meromorphic maps $f:X\dashrightarrow X$ of a compact complex variety $X$. Our aim is to obtain analogs of the above mentioned results for this case.  There are several equivalent notions of topological entropy for meromorphic maps. Here we use the one first defined by S. Friedland. The set $I_{\infty}(f):=\bigcup _{j\geq 0}f^{-j}(I(f))$ is a countable union of proper analytic subsets of $X$, and hence is nowhere dense in $X$. If $x\in X\backslash I_{\infty}(f)$, then $f^j(x)\notin I(f)$ for all $j=0,1,2,\ldots $. By Tykhonov's theorem, the countable Cartesian product $X^{\mathbb{N}}$ is a compact Hausdorff space, and it is moreover a compact metric space.  We define $\Gamma _{f,\infty}\subset X^{\mathbb{N}}$ to be the closure of the set $\{(x,f(x),f^2(x),\ldots ):~x\in X\backslash I_{\infty}(f)\}$. Then $\Gamma _{f,\infty}$ is itself a compact metric space, and we have an induced continuos map $\phi _f:\Gamma _{f,\infty}\rightarrow \Gamma _{f,\infty}$ given by the formula $\phi _f(x_1,x_2,x_3,\ldots )=(x_2,x_3,\ldots )$. The topological entropy of $f$ is then given by:
\begin{eqnarray*}
h_{top}(f):= h_{top}(\phi _f).
\end{eqnarray*} 
If $\mu$ is a probability measure on $X$ having no mass on proper analytic subsets and is invariant by $f$, then we can define exactly as in the continuous dynamics case a notion of measure entropy $h_{\mu}(f)$, and it is again true that $h_{\mu}(f)\leq h_{top}(f)$. However, if $\mu$ has mass on proper analytic subsets, then in general it is not known how to define measure entropy $h_{\mu}(f)$, since in the definition of measure entropy we need to use the fact that the preimages by $f^j$ of any Borel $\mu$-partition of $X$ are again Borel $\mu$-partitions of $X$ for all $j=0,1,2,\ldots$. The latter is no longer guaranteed if $f$ is merely meromorphic and $\mu$ has mass on proper analytic subsets of $X$. Also, theoretically the Variational Principle may not hold if we restrict to only measures having no mass on proper analytic subsets. Likewise, other properties of continuous dynamics, which we mention above, do not hold in the meromorphic setting. For example, there is no guarantee that the Cesaro's average approach can provide us with invariant measures.  See Section 5 for a discussion in the related case of transcendental maps of $\mathbb{C}$. 

However, as we will see next, still some analogs of the above classical results in continuous dynamics hold for meromorphic maps, if we allow positive strong submeasures in the consideration. The main idea is to relate the dynamics of $f$ and $\phi _f$ via the natural projection $\pi _1:\Gamma _{f,\infty}\rightarrow X$ given by the following formula $\pi _1(x_1,x_2,\ldots )=x_1$. It is easy to check that $\pi _1$ is continuous and surjective. The key is that while $\phi _f$ and $f$ are not semi-conjugate via $\pi _1$ in the strict sense (since $f$ is not continuous), they are almost semi-conjugate: If $\hat{x}\in \Gamma _{f,\infty}$ is such that $\pi _1(\hat{x})\notin I(f)$, then $\pi _1\circ \phi _f(\hat{x})=f\circ \pi _1(\hat{x})$. The precise relation between these two maps, in terms of measures, is the following simple but very useful observation. 
\begin{proposition}
If $\hat{\mu}$ is a positive strong submeasure on $\Gamma _{f,\infty}$, then 
\begin{eqnarray*}
f_* (\pi _1)_*(\hat{\mu})\geq (\pi _1)_*(\phi _f)_*(\hat{\mu}).
\end{eqnarray*}
In general, we have that $f_* (\pi _1)_*(\hat{\mu})\not= (\pi _1)_*(\phi _f)_*(\hat{\mu})$, even if $\hat{\mu}$ is a measure. 
\label{PropositionKeyEntropy}\end{proposition}

We now discuss some properties of invariant positive strong submeasures. We note that in the case $f$ is a continuous map and $\mu$ is a measure, then $f_*(\mu)=\mu$ if and only if either $f_*(\mu)\geq \mu$ or $f_*(\mu )\leq \mu$. In the case we are concerned here, when $f:X\dashrightarrow X$ is a dominant meromorphic map of a compact complex variety, and $\mu$ is a positive strong submeasure, the properties $f_*(\mu )=\mu$,  $f_*(\mu )\geq \mu$ and $f_*(\mu )\leq \mu$ are in general not the same. However, these properties are very much related. For example, it can be checked that if $\mu$ is a measure and $f_*(\mu)\leq \mu$, then $f_*(\mu)=\mu$. On the other hand, we will see that positive strong submeasures $\mu$ having the property that $f_*(\mu)\geq \mu$ appear very naturally in dynamics. If we apply Cesaro's average procedure for meromorphic maps, we obtain positive strong submeasures $\mu$ with $f_*(\mu)\geq \mu$. Likewise, if we have a positive strong submeasure $\hat{\mu}$ on $\Gamma _{f,\infty}$ which is $\phi _f$ - invariant, then $\mu =(\pi _1)_*\hat{\mu}$ satisfies $f_*(\mu)\geq \mu$. These, and many other properties of invariant positive submeasures will be proven in Theorem \ref{TheoremInvariantMeasures}. In particular, we obtain in the next result canonical invariant positive strong submeasures to those $\mu$ which satisfy either $f_*(\mu)\geq \mu$ or $f_*(\mu )\leq \mu$.  

\begin{theorem} Let $f:X\dashrightarrow X$ be a dominant meromorphic map of a compact complex variety. Let $0\not= \mu _0\in SM^+(X)$.

i) If $f_*(\mu _0)=\mu _0$, then there exists a non-zero measure $\hat{\mu _0}$ on $\Gamma _{f,\infty}$ so that $(\phi _f)_*(\hat{\mu _0})=\hat{\mu _0}$ and $(\pi _1)_*(\hat{\mu _0})\leq \mu _0$. Moreover, the set $\{\hat{\mu} \in SM^+(\Gamma _{f,\infty}):~(\pi _1)_*(\hat{\mu})\leq \mu ,~(\phi _f)_*(\hat{\mu} )= \hat{\mu}\}$ has a maximum, denoted by $Inv(\pi _1,\mu )$.

ii) If $f_*(\mu _0)\leq \mu _0$, then the set  $\{\mu \in SM^+(X):~ \mu \leq \mu _0,~f_*(\mu )= \mu \}$ is non-empty and has a largest element, denoted by $Inv(\leq \mu _0)$. Moreover, $Inv(\leq \mu _0)=\lim _{n\rightarrow\infty}(f_*)^n(\mu )$. 

iii) If $f_*(\mu _0)\geq \mu _0$, then the set $\{\mu \in SM^+(X):~ \mu \geq \mu _0,~f_*(\mu )= \mu \}$ is non-empty and has a smallest element, denoted by $Inv(\geq \mu _0)$. 

\label{TheoremInvariantMeasuresShortVersion}\end{theorem}

In summary, we have several canonical ways to associate invariant positive strong submeasures on $X$ or $\Gamma _{f,\infty}$. If $\hat{\mu}$ is an invariant positive strong submeasure on $\Gamma _{f,\infty}$, then we obtain an invariant positive strong submeasure $Inv(\geq (\pi _1)_*(\hat{\mu}))$ on $X$. Conversely, if $\mu$ is an invariant positive strong submeasure on $X$, then we obtain an invariant positive strong submeasure $Inv(\pi _1, \mu )$ on $\Gamma _{f,\infty}$. If $\mu$ is any positive strong submeasure on $X$ and $\mu _{\infty}$ is any cluster point of Cesaro's average $\frac{1}{n}\sum _{j=0}^n(f_*)^j(\mu _{\infty})$, then we obtain an invariant positive strong submeasure $Inv (\geq \mu _{\infty})$. 

Now we are ready to discuss entropy of positive strong submeasures invariant by a meromorphic map $f:X\dashrightarrow X$. We recall the discussion above that it is unknown how to define measure entropy for a measure having mass in the indeterminacy set $I(f)$, and also it is not guaranteed that dynamically interesting positive measures exist for $f$. In constrast, there are abundant invariant positive strong submeasures. Therefore, it is interesting to define the notion of measure entropy for these strong submeasures. As will be seen in Section 4, just as in the case of measures with mass in the indeterminacy set $I(f)$, a naive adaptation of the classical definition in this case will usually lead to the value $\infty$, which is pathological in view of the fact that topological entropy of a meromorphic map is finite (bounded from above by its dynamical degrees \cite{dinh-sibony3}) and the Variational Principle in continuous dynamics. 

To motivate a reasonable definition of measure entropy for meromorphic maps, we look at two special situations. The first one is discussed in Section 4.1, where we show that if for an invariant positive strong submeasure $\mu$ of a continuous map $f:X\rightarrow X$ we  define $h_{\mu}(f)$ to be $\sup _{\nu \in M^+(X),~f_*(\nu )=\nu, ~ \nu \leq \mu }h_{\nu}(f)$, then the Variational Principle is still satisfied. Taking supremum such as above is also compatible with the fact that the measure entropy $h_{\nu}(f)$ is monotone with respect to $\nu$. For the second special case, we consider now $f:X\dashrightarrow X$ a dominant meromorphic map, and $\mu$ an invariant positive measure with no mass on $I_{\infty}(f)$. Then we have the following simple observation that the measure entropy $h_{\mu}(f)$ can be computed using the continuous map $\phi _f$. 
\begin{proposition}
Let $f:X\dashrightarrow X$ be a dominant meromorphic map of a compact complex variety. Let $\mu$ be an invariant positive measure of $f$ with no mass on $I_{\infty}(f)$. Then there exists a unique measure $\hat{\mu}$ on $\Gamma _{f,\infty}$ so that $(\pi _1)_*(\hat{\mu})=\mu$ and $(\phi _f)_*(\hat{\mu})=\hat{\mu}$. Moreover,
\begin{eqnarray*}
h_{\mu}(f)=h_{\hat{\mu}}(\phi _f)=\sup _{\nu \in M^+(\Gamma _{f,\infty}),~(\phi _f)_*(\nu )=\nu ,~(\pi _1)_*(\nu )\leq \mu}h_{\nu}(\phi _f). 
\end{eqnarray*}
 \label{PropositionMotivationMeasureEntropy}\end{proposition}  

These two special cases suggest the following definition of measure entropy for positive strong submeasures invariant by meromorphic maps. 
\begin{definition}
Let $f:X\dashrightarrow X$ be a dominant meromorphic map of a compact complex variety $X$. Let $\mu$ be a positive strong submeasure invariant by $f$. We define 
\begin{eqnarray*}
h_{\mu}(f):=\sup _{\nu \in M^+(\Gamma _{f,\infty}), ~(\phi _f)_*(\nu )=\nu ,~ (\pi _1)_*(\nu )\leq \mu}h_{\nu }(\phi _f).
\end{eqnarray*}
\label{DefinitionMeasureEntropyMeromorphicMap}\end{definition}
We will show in Section 4 that this measure entropy has good properties as wanted (such as: monotone in $\mu$, bounded from above by dynamical degrees of $f$, and satisfies the Variational Principle). In particular, the following invariant positive strong submeasure captures the dynamics of $f$: We let $\hat{\mu}_{\phi _f,inv}$ to be the supremum of all $\phi _f$-invariant probability measures. Then $\hat{\mu}_{\phi _f,inv}$ is an $\phi _f$-invariant positive strong submeasure and $h_{top}(f)=h_{top}(\phi _f)=h_{\hat{\mu}_{\phi _f,inv}}(\phi _f)$. Now $\mu _{f,inv}=Inv(\geq (\pi _1)_*(\hat{\mu}_{\phi _f,inv}))$  is an $f$-invariant positive strong submeasure of mass $1$ and $h_{\mu _{f,inv}}(f)=h_{top}(f)$. By the remarks at the end of Section 4.1, we have that $\mu _{f,inv}<\sup _{x\in X}\delta _x$ in general, for example when $f$ is a holomorphic map which is not the identity map.

\subsection{Least negative intersection of positive closed $(1,1)$ currents}
Next we discuss an application to wedge intersection of positive closed $(1,1)$ currents on compact K\"ahler manifolds. Monge-Ampere operators of positive closed $(1,1)$ currents are a very classical and active area, both on compact and non-compact manifolds \cite{bedford-taylor, bedford-taylor2, demailly, fornaess-sibony, kolodziej, cegrell1, cegrell2, guedj-zeriahi, boucksom-eyssidieux-guedj-zeriahi}. While most of the works in the subject concerns Monge - Ampere operators with no mass on pluripolar sets, there are also several works with masses on pluripolar sets, for example \cite{lempert1, lempert, celik-poletsky, demailly2, lelong, zeriahi, xing, ahag-cegrell-czyz-pham}. For intersection of currents of higher bi-degrees, there are the methods of using super-potentials and tangent currents \cite{dinh-sibony, dinh-sibony1, dinh-sibony2} with applications to complex dynamics and geometry. In all of these works, the resulting (whenever well defined) is always a {\bf positive measure}. However, in any of these methods, there are cases where either it is not known how to define the intersection (for example if the Lelong numbers of the given currents are positive on a common set of positive dimension) or the cohomology class of the intersection is not the intersection of the corresponding cohomology classes. (This is the case of using non-pluripolar intersection in \cite{bedford-taylor2, guedj-zeriahi, boucksom-eyssidieux-guedj-zeriahi} or using residue currents in \cite{anderson, anderson-wulcan, anderson-blocki-wulcan}. For example, we can see that for the case in Example 3 in Section \ref{SectionIntersection}, both these two approaches give the answer $0$.) The latter property is desirable when discussing invariant measures of maps. We also remark that most of these approaches are {\bf local}, that is the intersection can be defined on small open sets and then patched together, while our approach in the below is of a global nature. 

In previous work \cite{truong1, truong2}, using regularisation of currents \cite{demailly, dinh-sibony3}, we can define a wedge intersection of positive closed currents whenever a continuity property (more precisely the existence and independence of limits of the wedge intersection of regularisations of the given positive closed currents) is satisfied. The resulting (whenever well defined) need not be positive measures, but is only a signed measure. However, not every wedge intersection of positive closed currents can be defined by that manner. 

Here, we use an idea different from previous works, and show that the Monge-Ampere operator of positive closed $(1,1)$ currents can be defined for {\bf all} positive closed $(1,1)$ currents - even if their intersection number in cohomology is {\bf negative} - if instead of signed measures we allow strong submeasures. The main idea is that in the classical case, when a wedge intersection of positive closed $(1,1)$ currents are defined, then Bedford-Taylor's monotone convergence is satisfied. In general, we do not have Bedford-Taylor's monotone convergence, but we can consider the set $\mathcal{G}$ of all cluster points (which are signed measures of the form $\chi ^+-\chi ^-$ where $\chi ^{\pm}$ are positive measures) obtained from all monotone approximations of the given currents with some control on the masses. The masses of all signed measures in $\mathcal{G}$ {\bf may not be bounded} (see Example 3 for more detail), but $\chi ^+(1)-\chi ^-(1)$ is a constant (indeed is the intersection number of the corresponding cohomology classes). Let $\mathcal{G}^*$ be the closure of $\mathcal{G}$ with respect to the weak convergence of signed measures. If we restrict to only signed measures in $\mathcal{G}^*$ whose negative part's mass $\chi ^{-}(1)$ is smallest then the corresponding masses {\bf are bounded}. We can then use (\ref{EquationExampleSubmeasure}) to obtain a strong submeasure, the so-called least negative intersection. More details will be given in Section \ref{SectionIntersection}.  The main result is Theorem \ref{TheoremLeastNegativeIntersection} which asserts among other things that this least negative intersection is symmetric and compatible with intersection in cohomology. Some more results about $ \Lambda (T_1,\ldots ,T_p,R)$, including criteria for when it is a positive strong submeasure or measure, will be given in Section \ref{SectionIntersection}.  We extract here one most important property of this least negative intersection. For the notations $\kappa _{T_1,T_2,\ldots ,T_p,R}$ in the statement of the result, the readers are referred to Section \ref{SectionIntersection}.
\begin{theorem} Let $X$ be a compact K\"ahler manifold of dimension $k$
Let $T_1,T_1',T_2,\ldots ,T_p$ be positive closed $(1,1)$ current on $X$, and $R$ a positive closed $(k-p,k-p)$ current. Assume that both $\kappa _{T_1,T_2,\ldots ,T_p,R}$, $\kappa _{T_1,T_2,\ldots ,T_p,R}\geq 0$. Then $\Lambda (T_1+T_1',T_2,\ldots ,T_p,R)\geq \Lambda (T_1,T_2,\ldots ,T_p,R)+\Lambda (T_1',T_2,\ldots ,T_p,R)$.
\label{TheoremLeastNegativeIntersectionShortVersion}\end{theorem}

Here we state one simple case where explicit calculations can be done.
\begin{proposition}
1) Let $X=\mathbb{P}^k$ be a projective space. Then the least negative intersection $\Lambda (T_1,\ldots ,T_p,R)$ in Theorem \ref{TheoremLeastNegativeIntersection} is always in $SM^+(X)$. 

2) Let $X=\mathbb{P}^2$, $D\subset \mathbb{P}^2$ be a line, and $[D]$ the current of integration on $D$. Then for all $\varphi \in C^0(X)$
\begin{eqnarray*}
\Lambda ([D],[D])(\varphi )=\sup _{D}\varphi . 
\end{eqnarray*}

3) Let $X=$ the blowup of $\mathbb{P}^2$ at a point $p$. Let $E$ be the exceptional divisor of the blowup, and $[E]$ the current of integration on $E$. Then $\{E\}.\{E\}=-1$ in cohomology, and for all $\varphi \in C^0(X)$

\begin{eqnarray*}
\Lambda ([E],[E])(\varphi )=\sup _{D}(-\varphi ).  
\end{eqnarray*}
\label{PropositionIntersectionProjectiveSpace}\end{proposition}
We will apply Theorem \ref{TheoremLeastNegativeIntersectionShortVersion} and Proposition \ref{PropositionIntersectionProjectiveSpace} to investigate the range of $\Lambda (T,T,\ldots ,T)$ for positive closed $(1,1)$ currents $T$ in Section \ref{SectionIntersection}. In particular, we propose the following conjectures, which roughly means that the range of the Monge-Ampere least negative intersection cannot contain too singular positive strong submeasures. 

{\bf Conjecture A.} Let $X$ be a compact K\"ahler manifold of dimension $k$. If there is a positive closed $(1,1)$ current $T$ on $X$ and a subset $A\subset X$ so that $\Lambda (T_1=T,T_2=T,\ldots ,T_k=T)\geq \sup _{x\in A}\delta _x$, then there is a proper subvariety $Z\subset X$ containing $A$. In particular, there is no positive closed $(1,1)$ current $T$ on $X$  so that $\Lambda (T_1=T,T_2=T,\ldots ,T_k=T)\geq \sup _{x\in X}\delta _x$. 

{\bf Conjecture B.} Let $X$ be a compact K\"ahler manifold, and let $\mu _1,\mu _2$ be two positive strong submeasures on $X$ so that $||\mu _1||=||\mu _2||$ and $\mu _1\geq \mu _2$. Assume that there is a positive closed $(1,1)$ current $T_1$ on $X$ so that $\Lambda (T_1,T_1,\ldots ,T_1)=\mu _1$. Then there is a positive closed $(1,1)$ current $T_2$ on $X$ so that $\Lambda (T_2,T_2,\ldots ,T_2)=\mu _2$.  

We end this subsection mentioning an application to dynamics of meromorphic maps. Let $f:X\dashrightarrow X$ be a dominant meromorphic map of a compact K\"ahler {\bf surface}. The study of dynamics of such maps is very active. It is now recognised that maps which are algebraic stable (those whose pullback on cohomology group is compatible with  iterates, that is $(f^n)^*=(f^*)^n$ on $H^{1,1}(X)$ for all $n\geq 0$) have good dynamical properties. An important indication of the complexity of such  maps is dynamical degrees defined as follows. Let $\lambda _1(f)$ be the spectral radius of the linear map $f^*:H^{1,1}(X)\rightarrow H^{1,1}(X)$ and let $\lambda _2(f)$ be the spectral radius of the linear map $f^*:H^{2,2}(X)\rightarrow H^{2,2}(X)$. There are two large interesting classes of such maps: those with large topological degree ($\lambda _2(f)>\lambda _1(f)$) and those with large first dynamical degree ($\lambda _1(f)>\lambda _2(f)$). The dynamics of the first class is shown in our paper \cite{dinh-nguyen-truong2} to be as nice as expected. For the second class, the most general result so far belongs to \cite{diller-dujardin-guedj1, diller-dujardin-guedj2, diller-dujardin-guedj3}, who showed the existence of canonical Green $(1,1)$ currents $T^+$ and $T^-$ for $f$, and who used potential theory to prove that the dynamics is nice (in particular, the wedge intersection $T^+\wedge T^-$ is well-defined as a positive measure) if the so-called finite energy conditions on the Green currents are satisfied. While these conditions are satisfied for many interesting subclasses, it is known however that in general they are false \cite{buff}. On the other hand, since it is known that $T^+$ has no mass on curves \cite{diller-dujardin-guedj1}, it follows from Theorem \ref{TheoremPositiveWedgeIntersection} that we have the following result. 
\begin{theorem}
Let $f:X\dashrightarrow X$ be a dominant meromorphic map of a compact K\"ahler surface which is algebraic stable and has $\lambda _1(f)>\lambda _2(f)$. Let $T^{+}$ and $T^{-}$ be the canonical Green $(1,1)$ currents of $f$. Then the least negative intersection $\Lambda (T^+,T^-)$ is in $SM^+(X)$. 
\label{TheoremDynamicsDimension2}\end{theorem}
At the moment we do not know whether $f_*(\Lambda (T^+,T^-))=\Lambda (T^+,T^-)$. However, by the discussion in the previous subsection, we can consider cluster points of Cesaro's average $\frac{1}{n}\sum _{j=1}^n(f_*)^j(\Lambda (T^+,T^-))$, and then obtain the associated invariant positive strong submeasures $\mu$. From such a $\mu$, we then can construct associated invariant measures of $\phi _f$. The significance these invariant positive strong submeasures and measures, together with their entropies, will be pursued in a future work. (See also Section 5 where we work out explicitly the cluster points of Cesaro's average, where we work with a transcendental holomorphic map of $\mathbb{C}^1$ and start with a probability measure $\mu $ on $\mathbb{P}^1$. In this case the only dynamically interesting invariant positive strong submeasure we can obtain is $\sup _{x\in \mathbb{P}^1}\delta _x$.) 

{\bf Remark.} We note a parallel between the least negative intersection and the tangent currents in this situation. Under the same assumptions as in Theorem \ref{TheoremDynamicsDimension2}, it was shown in our paper \cite{dinh-nguyen-truong} that the h-dimension (defined in \cite{dinh-sibony2}) between $T^+$ and $T^-$ is $0$, the best possible. 
 
\subsection{The case of transcendental maps on $\mathbb{C}$ and $\mathbb{C}^2$} 

In the discussion about entropy and invariant positive strong submeasures, the key property needed is that we have a compact metric space $X$, a dense open set $U\subset X$ and a continuous map $f:U\rightarrow X$ such that $f(U)$ contains a dense open subset of $X$. Therefore, the ideas we use before for entropy of meromorphic maps can be applied in much more general settings. For example, we can study these ergodic properties for transcendental maps on $\mathbb{C}$ and $\mathbb{C}^2$. 

Details will be given in Section 5. Here we just mention some main points. We would then consider $U=\mathbb{C}$ or $\mathbb{C}^2$ and compactifications $X$ of $U$. There are of course, many choices of $X$, but it is reasonable that we should choose $X$ to be a compact complex manifold. In the case $U=\mathbb{C}$, there is only one such choice, i.e.  $X=\mathbb{P}^1$. Moreover, $\mathbb{P}^1$ can be regarded as the simplest compactification of $\mathbb{C}$. In the case $U=\mathbb{C}^2$, we have many choices, that is $X=F_n$ a Hirzebruch's surface. Which $X$ is the correct one to choose? In fact, it is not obvious that $\mathbb{P}^2$, while in many aspects the simplest one geometrically, is always a good choice. For example, for the map $f(x,y)=(x+1,y^2)$, the entropy of the lift to $\mathbb{P}^2$ is $0$ (see e.g. \cite{guedj}) which is not the one we would like (the largest dynamical degree of $f$, which in this case is $\log 2$). It turns out that for this map, the correct compactification is $\mathbb{P}^1\times \mathbb{P}^1$. To resolve this, we will apply the ideas used in \cite{truong3} for entropy of meromorphic maps. 

\subsection{Why prefer submeasures?} In all our definitions above, we have used submeasures. However, there is also another natural generalisation of measures, that is supermeasures, which are the same as superlinear operators on $C^0(X)$. Which leads to the question: Why using submeasures but not supermeasures? We will now provide a brief explanation of why submeasures are better than supermeasures for the questions considered here. 

First, we look at the case of entropy for maps. In this case, submeasures go along very well with the Variational Principle. As we will see in Section 4, for any continuous map $f:X\rightarrow X$ of a compact metric space, there is a positive strong submeasure $\mu$ of mass $1$ so that $h_{\mu}(f)=h_{top}(f)$. If we use supermeasures instead, we would need to define measure entropy as follows: for a supermeasure $\mu$, we have $\tilde{h}_{\mu}:=\inf _{\nu \in M^+(X),~\nu \geq \mu}h_{\nu }(f)$. But then it is not true in general that there is a supermeasure $\mu$ of mass $1$ so that $\tilde{h}_{\mu}(f)=h_{top}(f)$. (In fact, there are examples of $f:X\rightarrow X$ such that $h_{top}(f)>h_{\nu}(f)$ for all invariant probability measure $\nu$.) Moreover, for a meromorphic map, it is not clear that the similar construction with either Cesaro's average procedure on $X$ or invariant measures on $\Gamma _{f,\infty}$ would lead to dynamically interesting invariant supermeasures. 

Second, we look at the case of intersection of positive closed $(1,1)$ currents. We see for example in Proposition \ref{PropositionIntersectionProjectiveSpace} that there is a clear and interesting answer using submeasures to self-intersection of the current of integration of some curves. If we use instead supermeasures, we will obtain rather uninteresting answers. For example, if $D\subset \mathbb{P}^2$ is a line, then the answer we obtain using supermeasures for the self-intersection of $[D]$ will be $\leq \inf _{x\in D}\varphi (x)$, and we do not know what is the exact answer.  Also, if we use supermeasures, then we do not know whether the $\Lambda (T^+,T^-)$ in \ref{TheoremDynamicsDimension2} will be positive, and hence whether it may be dynamically interesting. 

 \subsection{Organisation of the paper and Acknowledgments} In the next section we prove several basic properties of strong submeasures, together with results around the pullback and pushforward of strong submeasures by meromorphic maps and the observation that pullback and pushforward in (\ref{EquationSubmeasurePushforwardMeromorphic}) and (\ref{EquationSubmeasurePullbackMeromorphic}) can also be defined using any resolution of singularities of the graph $\Gamma _f$. In Section \ref{SectionIntersection}, we define least negative intersection of positive closed $(1,1)$ currents and prove Theorem \ref{TheoremLeastNegativeIntersection} and Proposition \ref{PropositionIntersectionProjectiveSpace} together with some other results. In Section \ref{SectionVariationalPrinciple}, we consider invariant positive strong submeasures of meromorphic maps and measure entropy. In Section 5 we apply these ideas to transcendental maps on $\mathbb{C}$ and $\mathbb{C}^2$. 

{\bf Acknowledgments.} The author would like to thank Lucas Kaufmann for an invitation for a visit to University of Gothenburg and the useful discussions there with him, Elizabeth Wulcan and David Witt Nystr\"om. We are grateful to Mattias Jonsson for suggesting the use of Hahn-Banach's theorem (Theorem \ref{TheoremHahnBanach}) as well as several other comments, which helped to greatly simplify the proofs and presentation of an earlier version of this paper. We would like also to thank Viet-Anh Nguyen for his comments and encouragement on \cite{truong}, and thank also to Jakob Hultgren for related discussions.

\section{Push-forward of positive strong submeasures by meromorphic maps}\label{SectionPullbackPushforward}
This section proves some basic properties of strong submeasures and results around pullback and pushforward by meromorphic maps. We will prove also the observation that pullback and pushforward in (\ref{EquationSubmeasurePushforwardMeromorphic}) and (\ref{EquationSubmeasurePullbackMeromorphic}) can also be defined using any resolution of singularities of the graph $\Gamma _f$. Some other results will also be proven, including Theorem \ref{TheoremPushforwardBlowup} in which we describe in more detail the pushforward map. 
 
First, we prove basic properties of strong submeasures.

 \begin{theorem} Let $X$ be a compact metric space. 

1) {\bf Weak-compactness.} Let $\mu _1,\mu _2,\ldots $ be a sequence in $SM(X)$ such that $\sup _{n}||\mu _n||<\infty$. Then there is a subsequence $\{\mu _{n(k)}\}_{k=1,2,\ldots }$ which weakly converges to some $\mu \in SM(X)$. If moreover $\mu _n\in SM^+(X)$, then so is $\mu$. 

2) If $\mu \in SM^+(X)$, then $||\mu ||=\max \{|\mu (1)|,|\mu (-1)|\}$. 

3) If $\mu _1,\mu _2\in SM(X)$ then $\max \{\mu _1,\mu _2\}$ and $\mu _1+\mu _2$ are also in $SM(X)$.  If $\mu _1,\mu _2\in SM^+(X)$ then $\max \{\mu _1,\mu _2\}$ and $\mu _1+\mu _2$ are also in $SM^+(X)$.
\label{TheoremSubmeasureBasicProperty}\end{theorem}
 \begin{proof}[Proof of Theorem \ref{TheoremSubmeasureBasicProperty}]

1) Since $X$ is a compact metric space, the space $C^0(X)$, equipped with the $L^{\infty}$ norm,  is separable. Therefore, there is a countable set $\varphi _1,\varphi _2,\ldots $ which is dense in $C^0(X)$. Because $\sup _{n}||\mu _n||=C<\infty$, for each $j$ the sequence $\{\mu _n(\varphi _j)\}_{n=1,2,\ldots }$ is bounded. Therefore, using the diagonal argument, we can find a subsequence $\{\mu _{n(k)}\}_{k=1,2,\ldots }$  so that for all $j$ the following limit exists:
\begin{eqnarray*}
\lim _{k\rightarrow\infty}\mu _n(\varphi _j)=: \mu (\varphi _j).
\end{eqnarray*}
As observed in the introduction, the fact that $\mu _n$ is sublinear and bounded implies that it is Lipschitz continuous: $|\mu _n(\varphi )-\mu _n(\psi )|\leq ||\mu _n||\times ||\varphi -\psi ||\leq C||\varphi -\psi ||$ for all $n$ and all $\varphi ,\psi \in C^0(X)$.  Then from the fact that $\{\varphi _j\}_{j=1,2,\ldots }$ is dense in $C^0(X)$, it follows that for all $\varphi \in C^0(X)$, the following limit exists: 
 \begin{eqnarray*}
\lim _{k\rightarrow\infty}\mu _n(\varphi )=: \mu (\varphi ).
\end{eqnarray*}
It is then easy to check that $\mu $ is also a strong submeasure, and if $\mu _n$ are all positive then so is $\mu$.

2) For any $\varphi \in C^0(X)$ we have $-||\varphi ||_{L^{\infty}}\leq \varphi \leq ||\varphi ||_{L^{\infty}}$. Therefore, since $\mu $ is positive, we have $\mu (-||\varphi ||_{L^{\infty}})\leq \mu (\varphi )\leq \mu (||\varphi ||_{L^{\infty}})$. By the sub-linearity of $\mu$, we have $\mu (-||\varphi ||_{L^{\infty}})=||\varphi ||_{L^{\infty}}\mu (-1)$ and $\mu (||\varphi ||_{L^{\infty}})=||\varphi ||_{L^{\infty}}\mu (1)$. Therefore, $|\mu (\varphi )|\leq ||\varphi ||_{L^{\infty}}\max \{|\varphi (-1)|,|\varphi (1)|\}$. Therefore, $||\mu ||\leq \max\{|\varphi (-1)|, |\varphi (1)|\}$. The reverse inequality follows from the fact that $||-1||_{L^{\infty}}=||1||_{L^{\infty}}=1$. 

3) This is obvious. 
\end{proof}

\begin{theorem} Let $X$ be a compact metric space and $\mu \in SM(X)$. Let $E(\mu ):BUS(X)\rightarrow [-\infty , \infty )$ be defined as in (\ref{EquationSubmeasureUpperSemicontinuous}).  Assume that $E(\mu )(0)$ is finite. We have: 

1) For all $\varphi \in BUS(X)$, the value $E(\mu )(\varphi )$ is finite. Moreover, $E(\mu )(0)=0$ and $E(\mu )(-1)\geq -\mu (1)$. 

2) {\bf Extension.} If $\mu$ is {\bf positive}, then for all $\varphi \in C^0(X)$ we have $E(\mu )(\varphi )=\mu (\varphi )$.  

3) Moreover, $E(\mu )$ satisfies the following properties 

 i) {\bf Sub-linearity.} $E(\mu) (\varphi _1+\varphi _2)\leq E(\mu )(\varphi _1)+E(\mu) (\varphi _2)$ and $E(\mu )(\lambda \varphi )=\lambda E(\mu )(\varphi )$ for $\varphi _1,\varphi _2,\varphi \in BUS(X)$ and a non-negative constant $\lambda$. 

ii ) {\bf Positivity.} $E(\mu )(\varphi _1)\geq E(\mu )(\varphi _2)$ for all $\varphi _1,\varphi _2\in BUS(X)$ satisfying $\varphi _1\geq \varphi _2$. 

iii) {\bf Boundedness.} There is a constant $C>0$ so that for all $\varphi \in BUS(X)$ we have $|E(\mu )(\varphi )|\leq C||\varphi ||_{L^{\infty}}$. The least such constant $C$ is in fact $||\mu ||$.

4) If $A_1,A_2$ are closed subsets of $X$ then $\mu (A_1\cup A_2)\leq \mu (A_1)+\mu (A_2)$. Likewise, if $B_1,B_2$ are open subsets of $X$ then $\mu (B_1\cup B_2)\leq \mu (B_1)+\mu (B_2)$.  
\label{TheoremSubmeasureUpperSemicontinuous}\end{theorem}
\begin{proof}[Proof of Theorem \ref{TheoremSubmeasureUpperSemicontinuous}]
1) We first observe that for all $\varphi \in C^0(X,\geq 0)$ then $\mu (\varphi )\geq 0$. (Note that as observed in Section 1, this fact alone does not imply that $\mu$ is a positive strong submeasure.) In fact, otherwise, there would be $\varphi _0\in C^0(X,\geq 0)$ so that $\mu (\varphi _0)<0$. Then by the definition of $E(\mu )$ and sublinearity of $\mu$ we have
\begin{eqnarray*}
E(\mu )(0)\leq \inf _{n\in \mathbb{N}}\mu (n\varphi _0)=\inf _{n\in \mathbb{N}}n\mu (\varphi _0)=-\infty , 
\end{eqnarray*}
which is a contradiction with the assumption that $E(\mu )(0)$ is finite. 

Therefore, if $\varphi \in C^0(X,\geq 0)$, we obtain
\begin{eqnarray*}
0\leq \inf _{\psi \in C^0(X,\geq \varphi )}\mu (\psi )\leq \mu (\varphi ).
\end{eqnarray*}
Therefore, for these functions $\varphi$ we have $E(\mu )(\varphi )$ is a finite number. In particular, $0\leq E(\mu )(1)\leq \mu (1)$.  

Next, we observe that if $\varphi _1,\varphi _2\in BUS(X)$ such that either $E(\mu )(\varphi _1)$ or $E(\mu )(\varphi _2)$ is finite, then the proof of 3i) is still valid and gives $E(\mu )(\varphi _1+\varphi _2)\leq E(\mu )(\varphi _1)+E(\mu )(\varphi _2)$. Apply this sub-linearity to $\varphi _1=\varphi _2=0$ we obtain $E(\mu )(0)=E(\mu) (0+0)\leq 2E(\mu )(0)$, which implies $E(\mu )(0)\geq 0$. On the other hand, $E(\mu )(0)\leq \mu (0)=0$. Therefore, $E(\mu )(0)=0$. 
 
Since $E(\mu )(1)$ is finite, applying the above sub-linearity for $\varphi _1=1$ and $\varphi _2=-1$, we obtain $0=E(0)=E(\mu )(1+(-1))\leq E(\mu )(1)+E(\mu )(-1)$. Therefore, $E(\mu )(-1)\geq -E(\mu )(1)\geq -\mu (1)$.  

Finally, applying the proof of part 3iii) we deduce that for all $\varphi \in BUS(X)$, the number $E(\mu )(\varphi )$ is finite. 

2) Let $\varphi \in C^0(X)$, and choose any $\psi \in C^0(X,\geq \varphi )$. Since $\mu $ is positive, we have by definition that $\mu (\psi )\geq \mu (\varphi )$. Since $\varphi $ is itself contained in $C^0(X,\geq \varphi )$, it follows that 
\begin{eqnarray*}
E(\mu )(\varphi )=\inf _{\psi \in C^0(X,\geq \varphi )}\mu (\psi )=\mu (\varphi ). 
\end{eqnarray*}

3) Let $\varphi , \varphi _1,\varphi _2\in BUS(X)$. 

i) If $\psi _1\in C^0(X,\geq \varphi _1)$ and $\psi _2\in C^0(X,\geq \varphi _2)$ then $\psi _1+\psi _2\in C^0(X,\geq \varphi _1+\varphi _2)$. Therefore, by sub-linearity of $\mu$:
\begin{eqnarray*}
E(\mu )(\varphi _1+\varphi _2)=\inf _{\psi \in C^0(X,\geq \varphi _1+\varphi _2)}\mu (\psi )\leq \mu (\psi _1 + \psi _2)\leq \mu (\psi _1)+\mu (\psi _2). 
\end{eqnarray*}
We can choose $\psi _1$ and $\psi _2$ so that $\mu (\psi _1)$ is arbitrarily close to $E(\mu )(\varphi _1)$ and $\mu (\psi _2)$ is arbitrarily close to $E(\mu )(\varphi _2)$, and from that obtain the desired conclusion $E(\mu )(\varphi _1+\varphi _2)\leq E(\mu )(\varphi _1)+E(\mu )(\varphi _2)$. The other part of i) is easy to check. 

ii) If $\varphi _1\geq \varphi _2$ then $C^0(X,\geq \varphi _1)\subset C^0(X,\geq \varphi _2)$. From this the conclusion follows. 

iii)  We observe that we can find $\psi \in C^0(X,\geq \varphi )$ so that $||\psi ||_{L^{\infty}}=||\varphi ||_{L^{\infty}}$, simply by defining $\psi =\max\{\min \{\psi _0, ||\varphi ||_{L^{\infty}}\},-||\varphi ||_{L^{\infty}}\} $ for any $\psi _0\in C^0(X,\geq \varphi )$. Then $$E(\mu )(\varphi )\leq \mu (\psi )\leq ||\mu ||\times ||\psi ||_{L^{\infty}}=||\mu ||\times ||\varphi ||_{L^{\infty}}.$$ 
By the positivity of $E(\mu )$ in ii), we have $$E(\mu )(\varphi )\geq ||\varphi  ||_{L^{\infty}}E(\mu )(-1 ),$$
and hence $|E(\mu )(\varphi )|\leq \max \{|E(\mu )(-1)|,||\mu ||\}=||\mu ||$. In the last equality we used that 1) and positivity imply $-\mu (1)\leq E(\mu ) (-1)\leq E(\mu )(0)=0$.   

4) By definition we have for closed subsets $A_1,A_2\subset X$ 
\begin{eqnarray*}
\mu (A)= E(\mu )(1_{A_1\cup A_2})\leq E(\mu )(1_{A_1}+1_{A_2})\leq E(\mu )(1_{A_1})+E(\mu )(1_{A_2})=\mu (A_1)+\mu (A_2). 
\end{eqnarray*}
In the first inequality we used $1_{A_1\cup A_2}\leq 1_{A_1}+1_{A_2}$ and the positivity of $E(\mu )$. In the second inequality we used the sub-linearity of $E(\mu )$. 

If $B_1,B_2$ are open subsets of $X$ and $A\subset B_1\cup B_2$ is closed in $X$, then since $X$ is compact metric we can find closed subsets $A_1,A_2$ of $X$ so that $A_1\subset B_1$,  $A_2\subset B_2$ and $A_1\cup A_2=A$. To this end, for each $x\in A$, we choose an open ball $B(x,r_x)$ (in the given metric on $X$) where $r_x>0$ is chosen as follows: if $x\in B_1$ then the closed ball $\overline{B(x,r_x)}$ belongs to $B_1$, if $x\in B_2$ then the closed ball $\overline{B(x,r_x)}$ belongs to $B_2$, and if $x\in B_1\cap B_2$ then the closed ball $\overline{B(x,r_x)}$ belongs to $B_1\cap B_2$. Since $A$ is compact, there is a finite number of such balls covering $A$: $A\subset \bigcup _{i=1}^mB(x_i,r_i)$. Then the choice of $A_1=A\cap (\bigcup _{x_i\in B_1} \overline{B(x_i,r_i)})$  and $A_2=A\cap (\bigcup _{x_i\in B_2} \overline{B(x_i,r_i)})$ satisfies the requirement. Then from the above sub-linearity of $\mu $ for compact sets and the definition, we have also sub-linearity for open sets $\mu (B_1\cup B_2)\leq \mu (B_1)+\mu (B_2)$.   

\end{proof}

\begin{proof}[Proof of Proposition \ref{PropositionUpperSemicontinuousExtension}] 

1) Since $g$ is bounded, it is clear that $E(g)$ is also bounded. By definition, it is clear that $E(g)|_U=g$. 

Next we show that $E(g)$ is upper-semicontinuous. If $x\in U$, then there is a small ball $B(x,r)\subset U$, and hence it can be seen that
\begin{eqnarray*}
\limsup _{y\in X\rightarrow x}E(g)(y)=\limsup _{y\in U\rightarrow x}E(g)(y)=\limsup _{y\in U\rightarrow x}g(y)\leq g(x),
\end{eqnarray*}
since $g$ is upper-semicontinuous on $U$. 

It remains to check that if $x\in X\backslash U$, and $x_n\in X\rightarrow x$ then 
\begin{eqnarray*}
\limsup _{n\rightarrow \infty}E(g)(x_n)\leq E(g)(x).
\end{eqnarray*}
To this end, choose $y_n\in U$ so that $d(y_n,x_n)\leq 1/n$ (here $d(.,.)$ is the metric on $X$) and $|g(y_n)-E(g)(x_n)|\leq 1/n$ for all $n$. Then $y_n\rightarrow x$, and hence
\begin{eqnarray*}
\limsup _{n\rightarrow\infty}E(g)(x_n)=\limsup _{n\rightarrow\infty}g(y_n)\leq E(g)(x). 
\end{eqnarray*} 

2) Using the assumption that $g$ is continuous on $U$ and $U_1\subset U$, we first check easily that $E(g_1)|_{U}=g=E(g)|_{U}$, and then the equality $E(g_1)=E(g)$ on the whole of $X$. 

3) Finally, note that if $g_1$ and $g_2$ are two upper-semicontinuous functions then $g_1+g_2$ is also upper-semicontinuous, and the inequality $E(g_1+g_2)\leq E(g_1)+E(g_2)$ follows from properties of limsup. 
\end{proof}

Now we turn to properties of the pullback and pushforward operators by meromorphic maps. We will first show that the upper-semicontinuous pushforwards $(\pi _{X,f})_*$ and $(\pi _{Y,f})_*$ on functions in $C^0(\Gamma _f)$ indeed produce upper-semicontinuous functions on $X$ and $Y$ as intended. Then we prove the observation that pullback and pushforward in (\ref{EquationSubmeasurePushforwardMeromorphic}) and (\ref{EquationSubmeasurePullbackMeromorphic}) can also be defined using any resolution of singularities of the graph $\Gamma _f$. Finally, we prove Proposition \ref{PropositionUpperSemicontinuousExtension} and Theorem \ref{TheoremSubmeasurePushforwardMeromorphic}, and give an example showing that the inequality in part 3) of Theorem \ref{TheoremSubmeasurePushforwardMeromorphic} may be strict in general. 

We recall the setting. We have a dominant meromorphic map $f:X\dashrightarrow Y$, its graph $\Gamma _f\subset X\times Y$, and two induced projections $\pi _{X,f},\pi _{Y,f}:~\Gamma _f\rightarrow X,Y$. We let $U\subset X$ be a Zariski  open dense set so that $\pi _{U,f}=\pi _{X,f}|_{\pi _{X,f}^{-1}(U)}:~\pi _{X,f}^{-1}(U)\rightarrow U$ is a finite covering. If $\dim (X)=\dim (Y)$, then there exists similarly $W\subset Y$ a Zariski open dense set so that $\pi _{W,f}=\pi _{Y,f}|_{\pi _{Y,f}^{-1}(W)}:~\pi _{Y,f}^{-1}(W)\rightarrow U$ is a finite covering.  

Let $\varphi \in C^0(\Gamma _f)$. Then obviously $(\pi _{U,f})_*(\varphi )$ is a continuous function on $U$, and hence upper-semicontinuous. Therefore, $(\pi _{X,f})_*(\varphi )=E((\pi _{U,f})_*(\varphi ))$ is an upper-semicontinuous on $X$. Part 2) of Proposition \ref{PropositionUpperSemicontinuousExtension} shows that $(\pi _{X,f})_*(\varphi )$ is independent of the choice of $U$. 

Now if $\dim (X)=\dim (Y)$, we can consider the upper-semicontinuous pushforward $(\pi _{Y,f})_*(\varphi )$.  Recall that for $y\in W$, we have
\begin{eqnarray*}
(\pi _{W,f})_*(\varphi )(y)=\sum _{z\in \pi _{W,f}^{-1}(y)}\varphi (z). 
\end{eqnarray*}
 Even though the map $\pi _{(W,f)}$ is not an isomorphism, it is a finite covering map of degree $\deg (f)$. Moreover, if $y_n\in Y\rightarrow y$, then when counted  {\bf with multiplicities} $\pi _{W,f}^{-1}(y_n)$ converge to $\pi _{W,f}^{-1}(y)$. Therefore we also have that $(\pi _{W,f})_*(\varphi )$ is continuous on $W$, and again $(\pi _{Y,f})_*(\varphi )=E((\pi _{W,f})_*(\varphi )(y))$ is upper-semicontinuous. Again, part 2) of Proposition \ref{PropositionUpperSemicontinuousExtension} shows that $(\pi _{Y,f})_*(\varphi )$ is independent of the choice of $W$. 

Now let $\Gamma $ be any compact complex variety with surjective holomorphic maps $\pi :~\Gamma \rightarrow X$ and $h:~\Gamma \rightarrow Y$, so that $\pi $ is bimeromorphic and $f=h\circ \pi ^{-1}$. Let $U_1\subset X$ be a Zariski open dense set so that $\pi _{U_1}=\pi :\pi ^{-1}(U_1)\rightarrow U_1$ is a finite covering, and $W_1\subset X$ be a Zariski open dense set so that $h _{W_1}=h:~h ^{-1}(W_1)\rightarrow W_1$ is a finite covering. We have $(h_{W_1})_*\pi ^*(\varphi )$ is continuous on $W_1$ and $(\pi _{W,f})_*\pi _{X,f}^*(\varphi )$ is continuous on $W$, and they are the same on $W_1\cap W$ which is itself a Zariski open dense set.  Using part 2) of Proposition \ref{PropositionUpperSemicontinuousExtension} again, we have that  $E((h_{W_1})_*\pi ^*(\varphi ))$ and $E((\pi _{W,f})_*\pi _{X,f}^*(\varphi ))$ are the same.  

The above observation shows that the value $f_*(\mu )(\varphi )$ can also be defined when we use any such $\Gamma$, for example any resolution of singularities of the graph $\Gamma _f$. Similarly, we can also show that the value $f^*(\mu )(\varphi )$ can also be defined when we use any such $\Gamma$.  

\begin{theorem} Let $f:X\dashrightarrow Y$ be a dominant meromorphic map of compact complex varieties. Let $g: Y\dashrightarrow Z$ be another dominant meromorphic map of compact complex varieties. When considering pullback of strong positive submeasures we will assume moreover that $X$ and $Y$ have the same dimension. 

1) We have $f^*(SM^+(Y))\subset SM^+(X)$ and $f_*(SM^+(X))\subset SM^+(Y)$. Moreover, if $\mu \in SM^+(X)$ and $\nu \in SM^+(Y)$, then $f_*(\mu )(\pm 1)=\mu (\pm 1)$ and $f^*(\nu )(\pm 1)=\deg (f)\nu (\pm 1)$. Here $\deg (f)$ is the topological degree of $f$, that is the number of inverse images by $f$ of a generic point in $Y$.  In particular, $||f_*(\mu )||=||\mu ||$ and $||f^*(\nu )||=\deg (f)||\nu ||$. 

2) For all $\mu \in SM^+(X)$ we have $g_*f_*(\mu )\geq (g\circ f)_*(\mu )$. If $f$ and $g$ are {\bf holomorphic}, then equality happens. 
 
3) For all $\mu \in SM^+(Z)$ we have $(g\circ f)^*(\mu )\leq f^*g^*(\mu )$. If moreover, $f$ and $g$ are {\bf holomorphic}, and $g$ is {\bf bimeromorphic}, then equality happens.  

4) If $\mu _n\in SM^+(X)$ weakly converges to $\mu$,  and $\nu$ is a cluster point of $f_*(\mu _n)$, then $\nu \leq f_*(\mu )$. If $f$ is holomorphic, then $\lim _{n\rightarrow\infty}f_*(\mu _n)=f_*(\mu )$.  

5) If $\mu$ is a positive measure without mass on $I(f)$, then $f_*(\mu )$ is the same as the usual definition. 

6) For any positive strong submeasure $\mu$, we have $f_*(\mu )=\sup _{\chi \in \mathcal{G}(\mu)}f_*(\chi )$, where $\mathcal{G}(\mu )=\{\chi :$ $\chi $ is a measure and $\chi \leq \mu \}$. 

7) For every positive strong submeasures $\mu _1,\mu _2$, we have $f_*(\mu _1+\mu _2)\geq f_*(\mu _1)+f_*(\mu _2)$. 
\label{TheoremSubmeasurePushforwardMeromorphic}\end{theorem}
\begin{proof}[Proof of Theorem \ref{TheoremSubmeasurePushforwardMeromorphic}]

1) We will give the proof only for the pullback operator in (\ref{EquationSubmeasurePullbackMeromorphic}), since the proof for the pushforward is similar. Let $\nu \in SM^+(Y)$, we will show that $f^*(\nu )\in SM^+(X)$. Let $\varphi ,\varphi _1,\varphi _2\in C^0(X)$ and $0\leq \lambda \in \mathbb{R}$. 

First, we show that $f^*(\nu )(\pm 1)=\deg (f)\nu (\pm 1)$. In fact, it follows from the definition that $(\pi _{W,f})_*\pi _{X,f}^*(\pm 1)=\pm \deg (f)$ on $W$, and hence $$ (\pi _{Y,f})_*\pi _{X,f}^*(\pm 1)=E((\pi _{W,f})_*\pi _{X,f}^*(\pm 1))=\pm \deg (f)$$
on X as well. Then we have
\begin{eqnarray*}
f^*(\nu )(\pm 1)=\inf _{\psi \in C^{0}(Y,\geq \pm \deg (f))}\nu (\psi )=\nu (\pm \deg (f))=\deg (f)\nu (\pm 1). 
\end{eqnarray*}
In the first equality we used that $\nu $ is positive, and in the second equality we used that $\nu$ is sub-linear. 

Second, we show the positivity of $f^*(\mu )$. If $\varphi _1\geq \varphi _2$, then it can see from the definition that $(\pi _{Y,f})_*\pi _{X,f}^*( \varphi _1)\geq (\pi _{Y,f})_*\pi _{X,f}^*( \varphi _2)$. Therefore $C^0(Y,\geq (\pi _{Y,f})_*\pi _{X,f}^*( \varphi _1))\subset C^0(Y,\geq (\pi _{Y,f})_*\pi _{X,f}^*( \varphi _2))$, and hence it follows from definition that $f^*(\nu )(\varphi _1)\geq f^*(\nu )(\varphi _2)$. 

Next, we show that $f^*(\nu )$ is bounded and moreover $||f^*(\nu )||=\deg (f)||\nu ||$. By positivity of $f^*(\nu )$, we have $f^*(\nu )(-||\varphi ||_{L^{\infty}})\leq f^*(\nu )\leq f^*(\nu )(||\varphi ||_{L^{\infty}})$. Hence $f^*(\nu )$ is bounded, and we conclude by Theorem \ref{TheoremSubmeasureBasicProperty}. 

Finally, we show the sub-linearity. The equality $f^*(\nu )(\lambda \varphi )=\lambda f^*(\nu )(\varphi )$, for $\lambda \geq 0$, follows from the fact that $(\pi _{Y,f})_*\pi _{X,f}^*(\lambda \varphi )=\lambda (\pi _{Y,f})_*\pi _{X,f}^*(\varphi )$ and properties of infimum. We now prove that $f^*(\nu )(\varphi _1+\varphi _2)\leq f^*(\nu )(\varphi _1)+f^*(\nu )(\varphi _2)$. In fact, from Proposition \ref{PropositionUpperSemicontinuousExtension} we have $$(\pi _{Y,f})_*\pi _{X,f}^*(\varphi _1+\varphi _2)\leq (\pi _{Y,f})_*\pi _{X,f}^*( \varphi _1)+(\pi _{Y,f})_*\pi _{X,f}^*( \varphi _2),$$
and hence if $\psi _1\in C^0(Y,\geq (\pi _{Y,f})_*\pi _{X,f}^*( \varphi _1))$ and $\psi _2\in C^0(Y, \geq (\pi _{Y,f})_*\pi _{X,f}^*( \varphi _2))$ then $\psi _1+\psi _2\in C^0(Y,\geq (\pi _{Y,f})_*\pi _{X,f}^*(\varphi _1+\varphi _2))$. Hence, by definition 
\begin{eqnarray*}
f^*(\nu )(\varphi _1+\varphi _2)\leq \nu (\psi _1+\psi _2)\leq \nu (\psi _1) + \nu (\psi _2). 
\end{eqnarray*}
In the second inequality we used the sub-linearity of $\nu$. If we choose $\psi _1$ and $\psi _2$ so that $\nu (\psi _1)$ is close to $f^*(\nu )(\varphi _1)$ and $\nu (\psi _2)$ is close to $f^*(\nu )(\varphi _2)$, then we see that $f^*(\nu )(\varphi _1+\varphi _2)\leq f^*(\nu )(\varphi _1)+f^*(\nu )(\varphi _2)$ as wanted. 

 2) By definition, we have
 \begin{eqnarray*}
 (g\circ f)_*(\mu )(\varphi )=\inf _{\psi \in C^0(X,\geq (g\circ f)^*(\varphi ))}\mu (\psi ). 
 \end{eqnarray*}
 Here we recall that $(g\circ f)^*(\varphi )$ is the upper-semicontinuos pullback of $\varphi $ by $g\circ f$. 
 
 On the other hand,
 \begin{eqnarray*}
g_*f_*(\mu )(\varphi )=\inf _{\psi _1\in C^0(Y,\geq g^*(\varphi ))}f_*(\mu )(\psi _1)=\inf _{\psi _1\in C^0(Y,\geq g^*(\varphi ))}\inf _{\psi _2\in C^0(X,\geq f^*(\psi _1))} \mu (\psi _2). 
 \end{eqnarray*}

 Then, since there is a dense Zariski open set $U\subset X$ so that all maps $f:U\rightarrow f(U)$ and $g:f(U)\rightarrow (g\circ f)(U)$ are all covering maps with finite fibres, it follows by the proof of part 2) of Proposition \ref{PropositionUpperSemicontinuousExtension} that whenever $\psi _1\in C^0(Y,\geq g^*(\varphi ))$  and $\psi _2\in C^0(Y,\geq g^*(\psi _1 ))$, then $\psi _2\in C^0(X, \geq (g\circ f)^*(\varphi ))$. From this, we get $g_*f_*(\varphi )(\nu )\geq (g\circ f)_*(\mu )(\varphi )$. 
 
When $f$ and $g$ are both holomorphic and $\varphi $ is continuous, then $g^*(\varphi )$, $f^*(g^*(\varphi ))$ and $(g\circ f)^*(\varphi )$ are all continuous functions. Then using the positivity of $\mu$, we can easily see that $$(g\circ f)_*(\mu )(\varphi )=\mu (f^*g^*(\varphi ))=f_*(\mu )(g^*(\varphi ))=g_*f_*(\mu )(\varphi ).$$  

3) In fact, the same proof as in 2) implies that $f^*g^*(\mu )\geq (g\circ f)^*(\mu )$. Now we prove the converse, under the assumption that both $f$ and $g$ are holomorphic, and $g$ is bimeromorphic. By definition, for all $\varphi \in C^0(X_2)$ 
\begin{eqnarray*}
f^*g^*(\mu )(\varphi )=\inf _{\psi \in C^0(X_1,\geq f_*(\varphi ))}g^*(\mu )(\psi ). 
\end{eqnarray*}

Choose $\psi _0\in C^0(X,\geq (g\circ f)_*(\varphi ))$ so that $\mu (\psi _0)$ is close to $(g\circ f)^*(\mu )(\varphi )$. Then, since  on a Zariski dense open set of $X_1$ where $f_*(\varphi )$ is continuous we have $g^*((g\circ f)_*(\varphi ))=f_*(\varphi )$, and $g^*(\psi _0)$ is continuous, it follows that $g^*(\psi _0)\geq f_*(\varphi )$. Therefore, we have
\begin{eqnarray*}
f^*g^*(\mu )(\varphi )\leq g^*(\mu )(g^*(\psi _0)). 
\end{eqnarray*}
Note that since $g$ is a bimeromorphic holomorphism and $\psi _0$ is continuous, the upper-semicontinuous pushforward $g_*(g^*(\psi _0))$ is equal to $\psi _0$. Hence, from definition and the positivity of $\mu$ we have
\begin{eqnarray*}
g^*(\mu )(g^*(\psi _0))=\mu (\psi _0). 
\end{eqnarray*}
Combining the above inequalities, we get that $f^*g^*(\mu )(\varphi )\leq (g\circ f)^*(\mu )(\varphi )$. Hence, we get the desired conclusion $f^*g^*(\mu )(\varphi )= (g\circ f)^*(\mu )(\varphi )$.

4)  It is enough to show the following: for all $\varphi \in C^0(Y)$ then 
\begin{eqnarray*}
\limsup _{n\rightarrow \infty}\inf _{\psi \in C^0(X,\geq f^*(\varphi ))}\mu _n(\psi )\leq \inf _{\psi \in C^0(X,\geq f^*(\varphi ))}\mu (\psi ). 
\end{eqnarray*}
If we choose $\psi _0\in  C^0(X,\geq f^*(\varphi ))$ so that $\mu (\psi _0)$ is close to $f_*(\mu )(\varphi )$, then from $\mu _n(\psi _0)\rightarrow \mu (\psi _0)$ we obtain the conclusion. 

If $f$ is holomorphic, then $f^*(\varphi )$ is itself a continuous function. Then it is easy to see  that $\lim _{n\rightarrow\infty}f_*(\mu _n)(\varphi )=f_*(\mu )(\varphi )$. 

5)  The upper-semicontinuous pullback $f^*(\varphi )$ of a function $\varphi \in C^0(Y)$ is continuous on the open set $U=X\backslash I(f)$. Therefore, by choosing a small open neighborhood $U_1$ of $I(f)$ and a partition of unity subordinate to $U$ and $U_1$, it is easy to find for any $U_2\subset\subset U$ a $\psi \in C^0(X,\geq f^*(\varphi ))$ so that $\psi |_{U_2}=f^*(\varphi )|_{U_2}$.  From this and the assumption that $\mu$ has no mass on $I(f)$,  the conclusion follows. 

6) By Hironaka's resolution of singularities, we can assume  without loss of generality that $X$ is smooth, and can find a smooth compact complex manifold $Z$, together with two surjective holomorphic maps $\pi :Z\rightarrow X$ and $h:Z\rightarrow Y$, so that $\pi$ is a finite composition of blowups at smooth centres and  $f=h\circ \pi ^{-1}$. Then we have for any positive strong submeasure $\mu$ that $f_*(\mu )=h_*\pi ^*(\mu )$. It is easy to check that the conclusion holds for $h$, and hence to prove the result it suffices to prove that $\pi ^*(\mu )=\sup _{\chi \in \mathcal{G}(\mu )}\pi ^*(\chi )$. By part 3), we only need to prove the conclusion in the case $\pi :Z\rightarrow X$ is the blowup at a smooth centre. 

Since $\pi ^*(\mu )\geq \pi ^*(\chi )$ for all $\chi \in \mathcal{G}(\mu )$, it follows that $\pi ^*(\mu )\geq \sup _{\chi \in \mathcal{G}(\mu )}\pi ^*(\chi )$. Now we will prove the reverse inequality. To this end, it suffices to show that for any measure $\chi '\leq \pi ^*(\mu )$, there is a measure $\chi \leq \mu$ so that $\chi '\leq \pi ^*(\chi )$. 

We first show that $\pi _*\pi ^*(\mu )=\mu $. In fact, if $\varphi \in C^0(X)$ then $\pi ^*(\varphi )\in C^0(Z)$ and $\varphi =\pi _*\pi ^*(\varphi )$. Hence, by definition 
\begin{eqnarray*}
\pi _*\pi ^*(\mu )(\varphi )=\pi ^*(\mu )(\pi ^*(\varphi ))=\mu (\pi _*\pi ^*(\varphi ))=\mu (\varphi ). 
\end{eqnarray*} 
Hence $\pi _*\pi ^*(\mu )=\mu $ as wanted. 

Now if $\chi '$ is any measure on $Z$, then $\chi =\pi _*(\chi ')$ is a measure on $X$. If moreover, $\chi '\leq \pi ^*(\mu )$, then
\begin{eqnarray*}
\chi =\pi _*(\chi ')\leq \pi _*\pi ^*(\mu )=\mu . 
\end{eqnarray*}
To conclude the proof, we will show that $\pi ^*(\chi )\geq \chi '$. To this end, let $\varphi \in C^0(Z)$, we will show that $\pi ^*(\chi )(\varphi )\geq \chi '(\varphi )$. By definition, the value of the positive strong submeasure $\pi ^*(\chi )$ at $\varphi$ is defined as: $\pi ^*(\chi )(\varphi )=\inf _{\psi \in C^0(X,\geq \pi _*(\varphi ))}\chi (\psi )$, and since $\chi =\pi _*(\chi ')$ the RHS is equal to $\inf _{\psi \in C^0(X,\geq \pi _*(\varphi ))}\chi '(\pi ^*(\psi ))\geq \chi '(\varphi )$. The latter follows from the fact that $\chi '$ is a positive strong submeasure and that for all $\psi \in C^0(X,\geq \pi _*(\varphi ))$ we have $\pi ^*(\psi )\geq \varphi $.   

7) Let $\varphi$ be a continuous function on $X$. We then have by part 6, using $\mathcal{G}(\mu _1)+\mathcal{G}(\mu _2)\subset \mathcal{G}(\mu _1+\mu _2)$, that
\begin{eqnarray*}
f_*(\mu _1+\mu _2)(\varphi )&=&\sup _{\nu \in \mathcal{G}(\mu _1+\mu _2)}\inf _{\psi \in C^0(\geq f^*(\varphi ))}\nu (\psi )\\
&\geq&\sup _{\nu _1\in \mathcal{G}(\mu _1),\nu _2\in \mathcal{G}(\mu _2)}\inf _{\psi \in C^0(\geq f^*(\varphi ))}(\nu _1+\nu _2) (\psi )\\
&\geq&\sup _{\nu _1\in \mathcal{G}(\mu _1),\nu _2\in \mathcal{G}(\mu _2)}[\inf _{\psi \in C^0(\geq f^*(\varphi ))}\nu _1 (\psi )+\inf _{\psi \in C^0(\geq f^*(\varphi ))}\nu _2 (\psi ) ]\\
&=&f_*(\mu _1)(\varphi )+f_*(\mu _2)(\varphi ).
\end{eqnarray*}
\end{proof}

{\bf Example 2.} Let $J:\mathbb{P}^2\dashrightarrow \mathbb{P}^2$ be the standard Cremona map given by $J[x_0:x_1:x_2]=[1/x_0:1/x_1:1/x_2]$. It is a birational map and is an involution: $J^2=$ the identity map. Let $e_0=[1:0:0]$, $e_1=[0:1:0]$ and $e_2=[0:0:1]$, and $\Sigma _i=\{x_i=0\}$ ($i=0,1,2$). Let $\pi :X\rightarrow \mathbb{P}^2$ be the blowup of $\mathbb{P}^2$ at $e_0,e_1$ and $e_2$, and let $E_0, E_1$ and $E_2$ be the corresponding exceptional divisors. Let $h=f\circ \pi :X\rightarrow \mathbb{P}^2$, then $h$ is a holomorphic map. Moreover, $\pi ^{-1}(e_0)=E_0$ and $h(E_0)=\Sigma _0$. More precisely, we have $h^{-1}(\Sigma _0\backslash \{e_1,e_2\})\subset E_0$. From this, we can compute, as in Example 1 in the introduction and proof of part 2) of Proposition \ref{PropositionUpperSemicontinuousExtension}, that for all $\varphi \in C^0(X)$ and for $\delta _{e_0}$ the Dirac measure at $e_0$
\begin{eqnarray*}
J_*(\delta _{e_0})(\varphi )=\sup _{\Sigma _0}\varphi . 
\end{eqnarray*}
It follows that $J_*(\delta _{e_0})\geq \max \{\delta _{e_1}, \delta _{e_2}\}$, where $\delta _{e_1}$ is the Dirac measure at $e_1$ and $\delta _{e_2}$ is the Dirac measure at $e_2$. Therefore, by the positivity of $J_*$ we obtain: 
\begin{eqnarray*}
J_*J_*(\delta _0 )(\varphi )\geq J_*(\max \{\delta _{e_1}, \delta _{e_2}\})(\varphi )\geq \max \{J_*(\delta _{e_1}(\varphi )), J_*(\delta _{e_2}(\varphi ))\}=\max \{\sup _{\Sigma _1}\varphi , \sup _{\Sigma _2}\varphi \}. 
\end{eqnarray*}
On the other hand, $J\circ J=$ the identity map, and hence $(J\circ J)_*(\delta _{e_0})=\delta _{e_0}$. Hence the inequality in part 2) of Theorem \ref{TheoremSubmeasurePushforwardMeromorphic} is strict in this case. Since $J=J^{-1}$, this example also shows that the inequality in part 3) of Theorem \ref{TheoremSubmeasurePushforwardMeromorphic} is strict in general. 

This example also shows that the inequality in part 4) of Theorem \ref{TheoremSubmeasurePushforwardMeromorphic} is strict in general. In fact, let $\{p_n\}\subset X\backslash I(f)$ be a sequence converging to a point $p=e_0$ and $\{J(p_n)\}$ converges to a point $q\in \Sigma _0$. Let $\mu _n=\max \{\delta _{p_1},\ldots ,\delta _{p_n}\}$. It can be checked easily that $\mu _n$ is an increasing sequence of positive strong submeasures, with $J_*(\mu _n)=\max \{\delta _{J(p_1)},\ldots ,\delta _{J(p_n)}\}$ for all $n$. Then the weak convergence limit $\mu =\lim _{n\rightarrow\infty}\mu _n=\sup _n\delta _{p_n}$ exists. In particular, $\mu \geq \delta _{e_0}$, and hence from the above calculation we find $J_*(\mu )\geq \sup _{x\in \Sigma _0}\delta _x$. On the other hand, $\nu =\lim _{n\rightarrow \infty}J_*(\mu _n)=\sup _{n}\delta _{J(p_n)}$. It is clear that if $x\in \Sigma _0\backslash \{q\}$, then $\nu$ cannot be compared with $\delta _x$. Therefore, we have the strict inequality $J_*(\mu )>\nu$ in this case.  

If we choose a sequence of points $\{p_n\}_{n=1,2,\ldots }\subset X\backslash \{e_0,e_1,e_2\}$ converging to $e_0$ and such that $q_n=J(p_n)$ converges to a point $q_0\in \Sigma _0$, then it can be seen that for $\mu =\sup _n\delta _{p_n}$ we have $J_*(\mu )>\sup _nJ_*(\delta _{p_n})$. Hence part 6) of Theorem \ref{TheoremSubmeasurePushforwardMeromorphic} does not hold in general if we replace $\mathcal{G}(\mu )$ by the smaller set $\mathcal{G}=\{\delta _{p_n}\}_n$. 
  
We end this section describing in detail the pushforward map on positive strong submeasures. The next result about a good choice of $\psi \in C^0(X,\geq \varphi )$ for some special bounded upper-semicontinuous functions will be needed for that purpose.
\begin{lemma} Let $X$ be a compact metric space, $A\subset X$ a closed set and $U=X\backslash A$. Let $\varphi$ be a bounded upper-semicontinuous function on $X$ so that $\chi =\varphi |_U$ is continuous on $U$ and $\gamma =\varphi |_A$ is continuous on $A$. For any $U'\subset\subset U$ an open set and $\epsilon >0$,  there is a function $\psi \in C^0(X,\geq \varphi )$ so that: 

i) $\psi |_{U'}=\chi $; ii) $\sup _A|\psi |_A-\gamma |\leq \epsilon$; and iii) $\sup _X|\psi |\leq \sup _X|\varphi |+\epsilon$. 
\label{LemmaGoodContinuousFunctionChoice}\end{lemma}
\begin{proof} Let $\epsilon _1>0$ be a small number to be determined later. Since $\varphi $ is upper-semicontinuous, for each $x\in A$, there is $r_x>0$, which we choose so small that $\overline{U'}\cap \overline{B(x,r_x)}=\emptyset$, so that 
\begin{eqnarray*}
\sup _{y\in U\cap \overline{B(x,r_x)}}\chi (y)\leq \gamma (x)+\epsilon _1.
\end{eqnarray*} 
Since $\gamma$ is continuous on $A$, by shrinking $r_x$ if necessary, we can assume that 
\begin{eqnarray*}
\sup _{x'\in A\cap \overline{B(x,r_x)}}|\gamma (x')-\gamma (x)|\leq \epsilon _1. 
\end{eqnarray*} 
Hence we obtain 
\begin{eqnarray*}
\sup _{y\in U\cap \overline{B(x,r_x)}}\chi (y)\leq \inf _{x'\in A\cap \overline{B(x,r_x)}}\gamma (x')+2\epsilon _1. 
\end{eqnarray*}
The function $\gamma |_{A\cap \overline{B(x,r_x)}}$ can be extended to a continuous function $\gamma _x$ on $\overline{B(x,r_x)}$. We can assume, by shrinking $r_x$ for example, that 
\begin{eqnarray*}
\sup _{x',x"\in \overline{B(x,r_x)}}|\gamma _x(x')-\gamma _x(x")|\leq \epsilon _1. 
\end{eqnarray*} 
Now, since $A$ is compact, we can find a finite number of such balls, say $B(x_1,r_1),\ldots ,B(x_m,r_m)$, which cover $A$. We choose $U"$ another open subset of $X$ so that $U'\subset\subset  U"\subset\subset U$ and so that $U",B(x_1,r_1),\ldots ,B(x_m,r_m)$ is a finite open covering of $X$. Let $\tau ,\tau _1,\ldots ,\tau _m$ be a partition of unity subordinate to this open covering. Then the function $$\psi (x)=\tau (x)\gamma (x)+\sum _{i=1}^m\tau _i(x)[\gamma _{x_i}(x)+4\epsilon _1],$$
with $4\epsilon _1<\epsilon$, satisfies the conclusion of the lemma.   
\end{proof}

Given $f:X\dashrightarrow Y$ a dominant meromorphic map between compact complex varieties and $\mu $ a positive strong submeasure on $X$. Without loss of generality, we can assume that $X$ is smooth, by using the above results. By part 6) of Theorem \ref{TheoremSubmeasurePushforwardMeromorphic} and its proof, to describe $f_*(\mu )$ it suffices to describe $\pi ^*(\mu )$ where $\pi :Z\rightarrow X$ is a blowup at a smooth centre and $\mu$ is a measure. The following result addresses this question. 

\begin{theorem}
Let $\pi :Z\rightarrow X$ be the blowup of $X$ at an irreducible smooth subvariety $A\subset X$. Let $\varphi \in C^0(Z)$.  Let $\mu$ be a positive measure on $X$, and decompose $\mu =\mu _1+\mu _2$ where $\mu _1$ has no mass on $A$ and $\mu _2$ has support on $A$. Then $\pi ^*(\mu _1)$ is a positive measure on $Z$, $\pi _*(\varphi )|_A$ is continuous, and we have
\begin{eqnarray*}
\pi ^*(\mu )(\varphi )=\pi ^*(\mu _1)(\varphi )+\mu _2(\pi _*(\varphi )|_A).
\end{eqnarray*}
Moreover, an explicit choice of the collection $\mathcal{G}$ in part 2) of Theorem \ref{TheoremHahnBanach} for $\pi ^*(\mu )$ will be explicitly described in the proof. 
\label{TheoremPushforwardBlowup}\end{theorem}
\begin{proof}
Let $B\in Z$ be the exceptional divisor of the blowup. Then $\pi :B\rightarrow A$ is a smooth holomorphic fibration, whose fibres are isomorphic to $\mathbb{P}^{r-1}$ where $r=$ the codimension of $A$. As in Example 1, it can be computed that for $x\in A$ then $\pi _*(\varphi )(x)=\sup _{y\in \pi ^{-1}(x)}\varphi $. Therefore, from what was said about the map $\pi :B\rightarrow A$, it follows that $\pi _*(\varphi )|_A$ is continuous. Then it is easy to see that the upper-semicontinuous function $\pi _*(\varphi )$ satisfies the conditions of Lemma \ref{LemmaGoodContinuousFunctionChoice}. It is easy to check that $\pi ^*(\mu _1)$ is a positive measure on $Z$.  Hence, by the conclusion of Lemma \ref{LemmaGoodContinuousFunctionChoice}, it is easy to see that 
\begin{equation}
\pi ^*(\mu )(\varphi )=\pi ^*(\mu _1)(\varphi )+\mu _2(\pi _*(\varphi )|_A).
\label{EquationDetail}\end{equation}

Now we provide  an explicit choice of the collection $\mathcal{G}$ associated to $\pi ^*(\mu )$ in part 2) of Theorem \ref{TheoremHahnBanach}. From Equation (\ref{EquationDetail}), it suffices to provide such a $\mathcal{G}$ for  $\mu _2$, since then the corresponding collection for $\mu$ will be $\pi ^*{\mu _1}+\mathcal{G}$. Therefore, in the remaining of the proof, we will assume that $\mu =\mu _2$ has support on $A$. 

Define $\psi (\varphi )=\pi _*(\varphi )|_A$, it then follows that  $\psi (\varphi )\in C^0(A)$, and 
\begin{eqnarray*}
\pi ^*(\mu )(\varphi )=\mu (\psi (\varphi )). 
\end{eqnarray*}
We now present an explicit collection $\mathcal{G}$ of positive measures on $X$ so that 
\begin{eqnarray*}
\mu (\psi (\varphi ))=\sup _{\chi \in \mathcal{G}}\chi (\varphi ). 
\end{eqnarray*}

To this end, let us consider for each finite open cover $\{U_i\}_{i\in I}$ of $A$, a partition of unity $\{\tau _i\}$ subordinate to the finite open cover $\{U_i\}$ of $A$, and local continuous sections $\gamma _i:U_i\rightarrow \pi ^{-1}(U_i)$, the following assignment on $B$: 
\begin{eqnarray*}
\chi (\{U_i\}, \{\tau _i\}, \gamma _i)(\varphi )=\mu (H(\varphi )), 
\end{eqnarray*}
 where $H(\varphi )\in C^0(A)$ is the following function 
 \begin{eqnarray*}
 H(\varphi )(x)=\sum _{i\in I}\tau _i(x)\varphi (\gamma _i(x)). 
 \end{eqnarray*}
 Since $H(\varphi )$ is linear and non-decreasing in $\varphi$, it is easy to see that $\chi $ is indeed a measure. Moreover, since $|H(x)|\leq \max _B|\varphi |$, it follows that $||\chi ||\leq ||\mu ||$.

We let $\mathcal{G}$ be the collection of such positive measures. We now claim that for all $\varphi \in C^0(Z)$
\begin{eqnarray*}
\mu (\psi (\varphi ))=\sup _{\chi \in \mathcal{G}}\chi (\varphi ). 
\end{eqnarray*}

We show first $\mu (\psi (\varphi ))\geq \chi (\varphi )$ for all $\chi \in \mathcal{G}$.  In fact, since $\gamma _i(x)\in \pi ^{-1}(x)$ for all $x\in A$, it follows by definition that 
\begin{eqnarray*}
H(\varphi )(x)\leq \sup _{\pi ^{-1}(x)}\varphi =\psi (\varphi )(x),
\end{eqnarray*}
for all $x\in A$. Hence $\mu (\psi (\varphi ))\geq \chi (\varphi )$. 
 
Now we show the converse. Let $\varphi $ be any continuous function on $Z$. Then for any $\epsilon >0$ we can always find a finite open covering $\{U_i\}_{i\in I}$ of $X$, depending on $\varphi$ and $\epsilon$, so that for all $x\in U_i$ we have
\begin{eqnarray*}
|\varphi (\gamma _i(x))-\sup _{\pi ^{-1}(x)}\varphi |\leq \epsilon. 
\end{eqnarray*} 
 It then follows that correspondingly $|H(x)-\psi (\varphi )(x)|\leq \epsilon$ for all $x\in A$. Therefore, for this choice of $\chi \in \mathcal{G}$
 \begin{eqnarray*}
 |\mu (\psi (\varphi ))-\mu (H)|\leq \epsilon ,
 \end{eqnarray*}
 and hence letting $\epsilon \rightarrow 0$  concludes the proof
\end{proof}

\section{Least-negative intersection of almost positive closed $(1,1)$ currents}\label{SectionIntersection}
In this section we define the least negative intersection of almost positive closed $(1,1)$ currents, more precisely of $p$ almost  positive closed $(1,1)$ currents and a positive closed $(k-p,k-p)$ current  on a compact K\"ahler manifold of dimension $k$. Then we prove Theorem \ref{TheoremLeastNegativeIntersection}, Proposition \ref{PropositionIntersectionProjectiveSpace} and some other results. At the end of the section we will discuss about the range of the least negative intersection $\Lambda (T,T,\ldots ,T)$. The content in this section is an improvement and extension of our previous preprint \cite{truong}. 

For background about positive closed currents and intersection theory the readers may consult the book \cite{demailly1}. Let $(X,\omega )$ be a K\"ahler manifold. We recall that an upper-semicontinuous function $u:X\rightarrow [-\infty ,\infty )$ is quasi-plurisubharmonic if there is a constant $A>0$ so that $A\omega  +dd^cu$ is a positive closed $(1,1)$ current. If $\Omega $ is any smooth closed $(1,1)$ form on $X$, the current $T=\Omega +dd^cu$ is almost positive. We then say that $u$ is a quasi-potential of $T$. For a quasi-psh function $u$, its Lelong number at $x\in X$ is defined as follows:
\begin{eqnarray*}
\nu (u,x)=\liminf _{z\rightarrow x}\frac{u(z)}{\log |z-x|}. 
\end{eqnarray*}
We then define  $\nu (T,x)=\nu (u,x)$. 

From the seminal work of Bedford and Taylor \cite{bedford-taylor, bedford-taylor2} on intersection of positive closed $(1,1)$ currents whose quasi-potentials are locally bounded, to later developments with important contributions from Demailly \cite{demailly}, Fornaess and Sibony \cite{fornaess-sibony}, Kolodziej \cite{kolodziej}, Cegrell \cite{cegrell1, cegrell2}, Guedj and Zeriahi \cite{guedj-zeriahi}, Boucksom, Eyssidieux, Guedj and Zeriahi \cite{boucksom-eyssidieux-guedj-zeriahi} and many others, the following monotone convergence is the cornerstone.  

\begin{definition} {\bf Monotone convergence.}  Let  $(X,\omega )$ be a K\"ahler manifold of dimension $k$, and $u_1,\ldots ,u_p$ be quasi-plurisubharmonic functions and $R$ a positive closed $(k-p,k-p)$. If $\{u_i^{(n)}\}_n$ $(i=1,2,3, \ldots )$ is a sequence of smooth quasi-psh functions decreasing to $u_i$, and there is a constant $B>0$ such that $dd^cu_i^{(n)}\geq -B\omega $ for all $i$ and $n$, then for any $1\leq i_1<i_2<\ldots <i_q\leq p$
\begin{eqnarray*}
\lim _{n\rightarrow\infty }dd^cu_{i_1}^{(n)}\wedge dd^cu_{i_2}^{(n)}\wedge \ldots \wedge dd^cu_{i_q}^{(n)}\wedge R
\end{eqnarray*}   
exists, and the limit does not depend on the choice of the regularisation $u_i^{(n)}$. We denote the limit by $dd^cu_{i_1}\wedge \ldots \wedge dd^cu_{i_q}\wedge R$. 
\end{definition}

Note that the existence of the regularisations $\{u_i^{(n)}\}_n$ in the above monotone convergence condition is classical in the local setting, and for a compact K\"ahler manifold the existence is proven by Demailly \cite{demailly}. We recall that a quasi-psh function $u$ has analytic singularities if locally it can be written as  
\begin{eqnarray*}
u=\gamma +c\log (|f_1|^2+\ldots +|f_m|^2),
\end{eqnarray*}
where $\gamma$ is a smooth function, $c>0$ is a positive constant, and $f_1,\ldots ,f_m$ are holomorphic functions. We will need the following theorem, see \cite{demailly}. 

\begin{theorem}
Let $(X,\omega )$ be a compact K\"ahler manifold. Let $T=\Omega +dd^cu$ be a positive closed $(1,1)$ current on $X$.

1) There are a sequence of smooth quasi-psh functions $u_n$ decreasing to $u$, and a constant $A>0$ (depending on $T$) so that $dd^cu_n\geq -A\omega $ for all $n$.   

2) There are a sequence of quasi-psh functions $u_n$ with {\bf analytic singularities} decreasing to $u$, and a sequence of positive numbers $\epsilon _n$ decreasing to $0$, so that the following 2 conditions are satisfied: i) $\Omega + dd^cu_n\geq -\epsilon _n\omega $ for all $n$; and ii) $\nu (u_n,x)$ increases to $\nu (u,x)$  uniformly with respect to $x\in X$. 
 
\label{TheoremDemaillyRegularisation}\end{theorem}

\begin{definition} {\bf Classical wedge intersection.} Let $(X,\omega )$ be a K\"ahler manifold and  $u_1,\ldots ,u_p$ quasi-psh functions on $X$. Assume that for every open subset $U\subset X$, the restrictions $u_i|_{U}$ ($i=1,\ldots ,p$) and $R|_U$ on the K\"ahler manifold $(U,\omega _U=\omega |_U)$ satisfy the monotone convergence property.  Then for every smooth closed $(1,1)$ forms $\Omega _1, \Omega _2,\ldots ,\Omega _p$ on $X$ and any open set $U\subset X$, the wedge intersection $(\Omega _1+dd^cu_1)|_U\wedge (\Omega _2+dd^cu_2)|_U\wedge \ldots (\Omega _p+dd^cu_p)|_U\wedge R|_U$ is well-defined by a corresponding monotone convergence. We say that in this case the intersection $(\Omega _1+dd^cu_1)\wedge (\Omega _2+dd^cu_2)\wedge \ldots (\Omega _p+dd^cu_p)\wedge R$ is {\bf classically defined}. 
\end{definition}

\begin{remark}
 1) By the seminal work \cite{bedford-taylor} classical wedge intersection is satisfied when $u_1,\ldots ,u_p$ are locally bounded. It is also true when the singularities of $u_1,\ldots ,u_p$ are "small" in a certain sense, see \cite{demailly, fornaess-sibony}.    
 
 2) If the above monotone convergence holds (for smooth regularisations $u_i^{(n)}$), then it also holds when we use more generally  {\bf bounded regularisations} $u_i^{(n)}$. See Lemma \ref{LemmaGoodMonotoneConvergence} for more details. 

3) In the approach using residue currents \cite{anderson, anderson-wulcan, anderson-blocki-wulcan}, for positive closed $(1,1)$ currents of analytic singularities, it was shown that while the monotone convergence does not hold for {\bf all bounded regularisations}, it does hold for special bounded regularisations such as $u_i^{(n)}=\max \{u_i,-n\}$. We note that this special class of regularisations is used in the non-pluripolar approach \cite{bedford-taylor2, guedj-zeriahi, boucksom-eyssidieux-guedj-zeriahi}. However, both these two approaches does not preserve cohomology classes, since they give the answer $0$ to Example 3 below. 
\end{remark}
 
The following notion will be frequently used in the remaining of the paper. 
\begin{definition} {\bf Good monotone approximation.} Let $(X,\omega)$ be a K\"ahler manifold, and $u$ a quasi-psh function on $X$. A sequence $\{u_n\}_{n=1,2,\ldots }$ of quasi-psh functions on $X$ is a good monotone approximation of $u$ if it satisfies the following two properties: 

i) There exists  $A>0$ so that $dd^cu_n\geq -A\omega $ for all $n$,

and 

ii) $u_n$ decreases to $u$. 
\end{definition}
 
 \begin{lemma}
 Let $u_1,\ldots ,u_p$ be quasi-psh functions on a K\"ahler manifold $(X,\omega )$, and $R$ a positive closed $(k-p,k-p)$ current on $X$. For each $i$, let $\{u_{i}^{(n)}\}$ be a good monotone  approximation for $u_i$. Assume that for each $n$, the intersection $dd^cu_1^{(n)}\wedge \ldots \wedge dd^cu_p^{(n)}\wedge R$ is classically defined. Then we can write $dd^cu_1^{(n)}\wedge \ldots \wedge dd^cu_p^{(n)}\wedge R=\mu _n^+-\mu _n^-$, where $\mu _n^{\pm}$ are positive measures on $X$ so that for any compact set $K\subset X$:
 \begin{eqnarray*}
 \sup _n||\mu _n^{\pm}||_K<\infty .
 \end{eqnarray*}
\label{LemmaSignedMeasuresBoundedMasses}\end{lemma}
\begin{proof}
For simplicity, we prove only for the case $p=1$. The case where $p>1$ is similar. 

By definition, there is $A>0$ so that $A\omega + dd^cu_1^{(n)}$ is a positive closed current on $X$ and $\mu _n^+=(A\omega +dd^cu_1^{(n)})\wedge R$ is a positive measure. Hence $dd^cu_1^{(n)}\wedge R=\mu _n^+-\mu _n^-$ where $\mu _n^-=A\omega \wedge R$ is also a positive measure on $X$. Note that $\mu _n^-$ is independent of $n$, and hence for any compact $K\subset X$: $\sup _n||\mu _n^-||_K<\infty$. The claim for $||\mu _n^+||_K$ follows from the assumption that $(A\omega +dd^cu_1^{(n)})\wedge R$ is classically defined.  
\end{proof} 
 
 Let $(X,\omega )$ be a compact K\"ahler manifold, and $R$ a positive closed $(k-p,k-p)$ current on $X$. We denote by $\mathcal{E}(R)$ the set of all tuples $(u_1,\ldots ,u_p)$ where $u_i$'s ($i=1,\ldots ,p$) are quasi-psh functions so that the intersection $dd^cu_1\wedge \ldots \wedge dd^cu_p\wedge R$ is classically defined. Thus $\mathcal{E}(R)$ is the largest class of tuples where Bedford - Taylor monotone convergence is still valid.  

Assume that $(u_1,\ldots ,u_p)\in \mathcal{E}(R)$. Then by definition, for any smooth closed $(1,1)$ forms $\Omega _1,\ldots ,\Omega _p$ on $X$, any open set $U\subset X$ and sequences of smooth functions $\{u_i^{(n)}\}$ defined on $U$ which are good monotone approximations of $u_i|_U$, we have 
\begin{eqnarray*}
&&\lim _{n\rightarrow\infty}(\Omega _1|_U+dd^cu_1^{(n)})\wedge (\Omega _2|_U+dd^cu_2^{(n)})\wedge \ldots \wedge (\Omega _p|_U+dd^cu_p^{(n)})\wedge R|_U\\
&=& (\Omega _1|_U+dd^cu_1|_U)\wedge (\Omega _2|_U+dd^cu_2|_U)\wedge \ldots \wedge (\Omega _p|_U+dd^cu_p|_U)\wedge R|_U.
\end{eqnarray*}
 
If in the above, we choose more generally $(u_1^{(n)},\ldots ,u_p^{(n)})$ not smooth, but only in $\mathcal{E}(R)$ for each $n$, then each term in the LHS of the above equality is still defined and the same equality occurs. That is, using more general good monotone approximations does not expand the set $\mathcal{E}(R)$. 

\begin{lemma} Let $(X,\omega )$ be a {\bf compact} K\"ahler manifold and $R$ a positive closed $(k-p,k-p)$ current on $X$. Let $u_1,\ldots ,u_p$ be quasi-psh functions on $X$ and $\Omega _1,\ldots ,\Omega _p$ smooth closed $(1,1)$ forms on $X$. For each $i$, let $\{u_i^{(n)}\}_n$ be a good monotone approximation (not necessarily smooth) of $u_i$. Assume that for each $n$, the tuple $(u_1^{(n)},\ldots ,u_p^{(n)})$ is in $\mathcal{E}(R)$, and the limit (in the weak convergence of signed measures) $\mu =\lim _{n\rightarrow\infty}(\Omega _1+dd^cu_1^{(n)})\wedge \ldots \wedge (\Omega _p+dd^cu_p^{(n)})\wedge R$ exists. Then, for each $i$ there is a {\bf smooth} good monotone approximation $\{v_i^{(n)}\}_n$ of $u_i$ and so that 
\begin{eqnarray*}
\lim _{n\rightarrow\infty}(\Omega _1+dd^cv_1^{(n)})\wedge \ldots \wedge (\Omega _p+dd^cv_p^{(n)})\wedge R=\mu .
\end{eqnarray*}
  
In particular, if $(u_1,\ldots ,u_p)\in \mathcal{E}(R)$ then for all open set $U\subset X$
\begin{eqnarray*}
&&\lim _{n\rightarrow\infty}(\Omega _1|_U+dd^cu_1^{(n)|_U})\wedge \ldots \wedge (\Omega _p|_U+dd^cu_p^{(n)}|_U)\wedge R|_U\\
&=&(\Omega _1|_U+dd^cu_1|_U)\wedge \ldots \wedge (\Omega _p|_U+dd^cu_p|_U)\wedge R|_U. 
\end{eqnarray*}

\label{LemmaGoodMonotoneConvergence}\end{lemma}
\begin{proof} For simplicity, we may assume that $\Omega _1=\ldots =\Omega _p=0$. 

By Theorem \ref{TheoremDemaillyRegularisation}, for each $i$ and $n$, there is a good monotone approximation $\{\Phi _m(u_i^{(n)})\}_m$ of $u_i^{(n)}$. 

Since $X$ is a compact metric space, the space $C^0(X)$ is separable. Therefore, there is a dense countable set $\mathcal{F}\subset C^0(X)$. We enumerate the elements in $\mathcal{F}$ as $\varphi _1,\varphi _2,\ldots $.

Since $\mu =\lim _{n\rightarrow\infty}dd^cu_1^{(n)}\wedge \ldots \wedge dd^cu_p^{(n)}\wedge R$, for each $l$ there is a number $n_l$ so that for all $\varphi \in \{\varphi _1,\ldots ,\varphi _l\}$ and for all $n\geq n_l$:
\begin{equation}
|\mu (\varphi )-dd^cu_1^{(n)}\wedge \ldots \wedge dd^cu_p^{(n)}\wedge R(\varphi )|\leq 1/l. 
\label{Equation1}\end{equation}
We can assume that $n_1<n_2<n_3\ldots $ Then for each $i$, the sequence $\{u_i^{(n_l)}\}_l$ is a good monotone approximation of $u_i$. Therefore, we can assume that $||u_i^{(n_l)}-u_i||_{L^1(X)}\leq 1/l^2$ for all $l$ and $i$. 
 
Since for each $l$ the tuple $(u_1^{(n_l)},\ldots ,u_p^{(n_l)})\in \mathcal{E}(R)$, and $\{\Phi _m(u_i^{(n_l)})\}_m$ is a smooth good monotone approximation of $u_i^{(n_l)}$, there is $m_l$ so that for all $m\geq m_l$ we have for all $\varphi \in \{\varphi _1,\ldots ,\varphi _m\}$:
\begin{equation}
|dd^c\Phi _m(u_1^{(n_l)})\wedge \ldots \wedge dd^c\Phi _m(u_i^{(n_l)}) (\varphi )-dd^cu_1^{(n_l)}\wedge \ldots \wedge dd^cu_p^{(n_l)}\wedge R(\varphi )|\leq 1/l. 
\label{Equation2}\end{equation}

Now we are ready to choose the sequence $\{v_i^{(l)}\}$. We first choose an intermediate sequence $w_i^{(l)}$.

Choose $w_i^{(1)}=\Phi _{m_1}(u_i^{n_1})$. We can also arrange so that $||w_i^{(1)}-u_i^{(n_1)}||_{L^1(X)}\leq 1$. 

Since $w_i^{(1)}$ is smooth, $w_i^{(1)}\geq u_i^{(n_1)}\geq u_i^{(n_2)}$, and $\{\Phi _m(u_i^{(n_2)})\}$ is a good monotone approximation of $u_i^{(n_2)}$, by Hartogs' lemma (see \cite{sadullaev}), for $m$ large enough we have $w_i^{(1)}+1\geq \Phi _m(u_i^{(n_2)})$ and (\ref{Equation2}) is satisfied. We choose $w_i^{(2)}=\Phi _m(u_i^{(n_2)})$ for one such $m$. 

Constructing inductively, we can find a sequence $w_i^{(l)}=\Phi _m(u_i^{(n_l)})$ for large enough $m$, so that (\ref{Equation2}) is satisfied and: i) $||w_i^{(l)}-u_i^{(n_l)}||_{L^1}\leq 1/l^2$, and ii) $w_i^{(l)}\leq w_i^{(l-1)}+1/l^2$.  

Now we define $$v_i^{(l)}=w_i^{(l)}+\sum _{h\geq l+1}\frac{1}{h^2}.$$

Then $v_i^{(l)}-v_i^{(l-1)}=w_i^{(l)} -w_i^{(l)}-1/l^2\leq 0$ by the choice of $w_i^{(l)}$. Moreover, $dd^cv_i^{(l)}=dd^cw_i^{(l)}$ and it is easy to check that $v_i^{(l)}$ decreases to $u_i$. Hence $\{v_i^{(l)}\}_l$ is a smooth good monotone approximation of $u_i$. 

Since $dd^cv_i^{(l)}\geq -A\omega $ for all $i$ and $l$, by Lemma \ref{LemmaSignedMeasuresBoundedMasses} it follows that after working with a subsequence of $\{v_i^{(l)}\}$ if needed, there is a signed measure $\mu '$ so that
\begin{eqnarray*}
\lim _{n\rightarrow\infty}dd^cv_1^{(n)}\wedge \ldots \wedge dd^cv_p^{(n)}\wedge R=\mu '.
\end{eqnarray*} 

Then from (\ref{Equation2}) we have that $\mu '(\varphi )=\mu (\varphi )$ for all $\varphi \in \mathcal{F}$. Since $\mathcal{F}$ is dense in $C^0(X)$, we have the desired conclusion. 
\end{proof}

Now we turn our attention to a general tuple $(u_1,\ldots ,u_p)$ of quasi-psh functions on a compact K\"ahler manifold $(X,\omega )$. By the very definition, if $(u_1,\ldots ,u_p)\notin \mathcal{E}(R)$, then it is not guaranteed that we can assign a unique (signed) measure to the (not yet defined) intersection $(\Omega _1+dd^cu_1)\wedge \ldots \wedge (\Omega _p+dd^cu_p)\wedge R$ using good monotone approximations. However, we can assign to the tuple $(\Omega _1+dd^c u_1,\ldots ,\Omega _p+dd^cu_p,R)$ a collection $\mathcal{G}$ of signed measures in the following manner. Let $\{u_i^{(n)}\}_n$ (for $i=1,\ldots ,p$) be a good monotone approximation of $u_i$. Assume also that all $u_i^{(n)}$ are smooth functions. Then by Lemma \ref{LemmaSignedMeasuresBoundedMasses}, we can write for each $n$
\begin{eqnarray*}
(\Omega _1+dd^cu_1^{(n)})\wedge \ldots \wedge (\Omega _p+dd^cu_p^{(n)})\wedge R=\mu _n^+-\mu _n^{-},
\end{eqnarray*}
where $\sup _{n}||\mu _n^{\pm}||<\infty$. Therefore, there is a subsequence $\{n(k)\}_k$ so that $\mu _{n(k)}^{\pm}$ weakly converges to positive measures $\mu ^{\pm}$, and the signed measure $\mu =\mu ^+-\mu ^-$ is one element in the collection $\mathcal{G}$. Below is the precise definition. 
\begin{definition}
Let $(X,\omega )$ be a compact K\"ahler manifold, $u_1,\ldots ,u_p$ quasi-psh functions, $\Omega _1,\ldots ,\Omega _p$ are smooth closed $(1,1)$ forms, and $R$ a positive closed $(k-p,k-p)$ current. Let $\mathcal{G}(u_1,\ldots ,u_p, \Omega _1,\ldots ,\Omega _p, R)$ be the set of signed measures $\mu$ of the form: 
\begin{eqnarray*}
\mu = \lim _{n\rightarrow\infty}(\Omega _1+dd^cu_1^{(n)})\wedge \ldots \wedge (\Omega _p+dd^cu_p^{(n)})\wedge R,
\end{eqnarray*}
where $\{u_i^{(n)}\}_n$ ($i=1,\ldots ,p$) is a good monotone approximation of $u_i$ by smooth quasi-psh functions. 
\label{DefinitionCollectionG}\end{definition}
Here are some properties of this collection. 
\begin{proposition} Let $(X,\omega )$ be a compact K\"ahler manifold. Let $u_1,\ldots ,u_p$ be quasi-psh functions, $\Omega _1,\ldots ,\Omega _p$ smooth closed $(1,1)$ forms and $R$ a positive closed $(k-p,k-p)$ current.  

1) Let  $u_1',\ldots ,u_p'$ be quasi-psh functions and $\Omega _1',\ldots ,\Omega _p'$ be smooth closed $(1,1)$ forms so that $\Omega _i+dd^cu_i=\Omega _i'+dd^cu_i'$ for all $i=1,\ldots ,p$. Then $$\mathcal{G}(u_1,\ldots ,u_p,\Omega _1,\ldots ,\Omega _p,R)=\mathcal{G}(u_1',\ldots ,u_p',\Omega _1',\ldots ,\Omega _p',R).$$ 

2) Let $\{u_i^{(n)}\}_{n}$ ($i=1,\ldots ,p$) be a good monotone approximation of $u_i$. Assume that for each $n$, the tuple $(u_1^{(n)},\ldots ,u_p^{(n)})$ belongs to $\mathcal{E}(R)$, and that there is a signed measure $\mu $ so that
\begin{eqnarray*}
\lim _{n\rightarrow\infty}(\Omega _1+dd^cu_1^{(n)})\wedge \ldots \wedge (\Omega _p+dd^cu_p^{(n)})\wedge R=\mu .
\end{eqnarray*}
Then $\mu \in \mathcal{G}(u_1,\ldots ,u_p, \Omega _1,\ldots ,\Omega _p,R)$.
\label{PropositionPropertyCollectionG}\end{proposition}
\begin{proof}
1) If we put $\phi _i=u_i-u_i'$, then $dd^c\phi _i=\Omega _i'-\Omega _i$, and hence $\phi _i$ is a smooth function. Whenever $\{u_i^{(n)}\}_n$ is a good monotone approximation of $u_i$ by smooth quasi-psh functions, then $\{u_i'^{(n)}\}_n=\{u_i^{(n)}-\phi _i\}_n$ is a good monotone approximation of $u_i'$ by smooth quasi-psh functions and vice versa. Moreover, for all $n$ and $i$
\begin{eqnarray*}
\Omega _i+dd^cu_i^{(n)}=\Omega _i'+dd^cu_i'^{(n)}. 
\end{eqnarray*}
Hence for all $n$, we have
\begin{eqnarray*}
(\Omega _1+dd^cu_1^{(n)})\wedge \ldots \wedge (\Omega _p+dd^cu_p^{(n)})\wedge R=(\Omega _1'+dd^cu_1'^{(n)})\wedge \ldots \wedge (\Omega _p'+dd^cu_p'^{(n)})\wedge R.
\end{eqnarray*}
From this, we obtain the conclusion. 

2) The proof is similar to that of Lemma \ref{LemmaGoodMonotoneConvergence}. 
\end{proof} 

From part 1) of Proposition \ref{PropositionPropertyCollectionG}, it follows that for any tuple of almost positive closed $(1,1)$ currents $(T_1,\ldots T_p)$ and a positive closed $(k-p,k-p)$ current $R$, there is a well-defined collection $\mathcal{G}(T_1,\ldots ,T_p,R)$ of signed measures, which reduces to one element in the case when the intersection $T_1\wedge \ldots \wedge T_p\wedge R$ is classically defined. Our main idea is to obtain one single object from this large collection, which is as close to positive measures as possible. An idea would be to define a strong submeasure using (\ref{EquationExampleSubmeasure}). There are some technical difficulties, since apparently this collection $\mathcal{G}(T_1,\ldots ,T_p,R)$ is a priori not bounded in mass, and hence the resulting strong submeasure might be $\infty$ identically. Also, this collection contains many signed measures of different natures, and so we need to have a criterion to say among two signed measures, which is closer to be a positive measure. We elaborate on this in the next paragraph, after an example illustrating these points.

{\bf Example 3.} Let $X$ be a compact K\"ahler surface and $D\subset X$ an irreducible curve. Let $T_1=T_2=[D]$, where $[D]$ is the current of integration over $D$. Then, the collection $\mathcal{G}(T_1,T_2)$ contains all signed measures of the form $\Omega \wedge [D]$, where $\Omega$ is an arbitrary smooth closed $(1,1)$ form cohomolgous to $[D]$. 
\begin{proof}
Let $\Omega $ be a smooth closed $(1,1)$ form cohomologous to $[D]$. Then we can write $T_1=T_2=\Omega +dd^cu$ for some quasi-psh function $u$. From part 2) of Proposition \ref{PropositionPropertyCollectionG}, we only need to construct two good monotone approximations $\{u_1^{(n)}\}_n$ and $\{u_2^{(n)}\}_n$ of $u$ by bounded quasi-psh functions so that 
\begin{eqnarray*}
\lim _{n\rightarrow\infty}(\Omega +dd^cu_1^{(n)})\wedge (\Omega +dd^cu_2^{(n)})=\Omega \wedge [D].
\end{eqnarray*}
We choose $u_1^{(n)}=\max \{u,-n\}$. Since $(\Omega +dd^cu_1^{(n)})\wedge [D]$ is classically defined for all $n$, we can find a good monotone approximation $\{u_2^{(n)}\}_n$ of $u$ by smooth quasi-psh functions so that
\begin{eqnarray*}
\lim _{n\rightarrow\infty}(\Omega +dd^cu_1^{(n)})\wedge (\Omega +dd^cu_2^{(n)})=\lim _{n\rightarrow\infty} (\Omega +dd^cu_1^{(n)})\wedge [D].
\end{eqnarray*}

Since $[D]$ is smooth outside of $D$, it follows that $u$ is smooth outside of $D$.  Moreover, since $[D]$ has Lelong number $1$ along $D$, it follows that 
$\lim _{x\rightarrow D}u(x)=-\infty$. Therefore, $u_1^{(n)}$ is $-n$ in an open neighborhood of $D$, and hence $dd^cu_1^{(n)}=0$ in a neighborhood of $D$. Therefore, for all $n$ we have $ (\Omega +dd^cu_1^{(n)})\wedge [D]=\Omega \wedge [D]$, and this proves the conclusion. 
\end{proof}

We first define an order on strong submeasures. If $\mu $ is a (positive) measure, then we define as usual its norm: $||\mu ||=\mu (1)$. If $\mu $ is a signed measure, we define its norm as follows: 

$||\mu ||=\inf\{||\mu ^+||+||\mu ^-||:~\mu =\mu ^+-\mu ^-$ where $\mu ^{\pm}$ are positive measures$\}$.

We then define a negative-part norm of $\mu$ as follows:

$||\mu ||_{neg}=\inf\{||\mu ^-||:~\mu =\mu ^+-\mu ^-$ where $\mu ^{\pm}$ are positive measures$\}$.

It is easy to check that $||\mu ||_{neg}=0$ iff $\mu $ is a positive measure. Hence, it is reasonable to say that a signed measure $\mu$ with $||\mu ||$ fixed is closer to be a (positive) measure if $||\mu ||_{neg}$ is smaller.  

\begin{definition} {\bf Partial order on strong submeasures.} Let $\mu _1=\sup _{\chi _1\in \mathcal{G}_1}\chi _1$ and $\mu _2=\sup _{\chi _2\in \mathcal{G}_2}\chi _2$ be two strong submeasures, where $\mathcal{G}_1$ and $\mathcal{G}_2$ be two collections of signed measures on $X$ so that $\sup _{\chi _1\in \mathcal{G}_1}||\chi _1||, \sup _{\chi _2\in \mathcal{G}_2}||\chi _2||<\infty$. We write $\mu _1\succ \mu _2$ if either 

i) $\inf _{\chi _1\in \mathcal{G}_1}||\chi _1||_{neg}<\inf _{\chi _2\in \mathcal{G}_2}||\chi _2||_{neg}$,

or

ii) $\inf _{\chi _1\in \mathcal{G}_1}||\chi _1||_{neg}=\inf _{\chi _2\in \mathcal{G}_2}||\chi _2||_{neg}$, and $\mu _1(\varphi )\geq \mu _2(\varphi )$ for all $\varphi \in C^0(X)$. 
 \label{DefinitionOrderOnMeasures}\end{definition}

We note that if $\mu =\mu ^+-\mu ^-\in \mathcal{G}(T_1,\ldots ,T_p,R)$, where $\mu ^{\pm}$ are positive measures then $\mu ^{+}(1)-\mu ^{-}(1)=\{T_1\}\cdots \{T_p\}\cdot \{R\}$, where the right hand side is the intersection in cohomology, is a {\bf constant}. Therefore, from the above arguments, we see that the rough idea to proceed is to use the construction (\ref{EquationExampleSubmeasure}) to signed measures in $\mathcal{G}(T_1,\ldots ,T_p,R)$, whose negative parts have mass of the smallest value. The only problem is that there may be no signed measure in $\mathcal{G}(T_1,\ldots ,T_p,R)$ whose negative part has the smallest value, and even so, there may be other signed measures whose negative parts have masses converging to the smallest value. Such limit signed measures should be considered as having the same roles as the signed measures in $\mathcal{G}(T_1,T_2,\ldots ,T_p,R)$ whose negative parts have masses of the smallest value. To cope with these issues, we proceed as follows. 

Let  the assumptions be as above. We define
\begin{equation}
\kappa (T_1,\ldots ,T_p,R)=\inf_{\mu \in \mathcal{G}(T_1,\ldots ,T_p,R)} ||\mu ||_{neg}.
\label{EquationNegativeNormOfG}\end{equation}
   
We also define $\mathcal{G}^*(T_1,\ldots ,T_p,R)$ to be the closure of $\mathcal{G}(T_1,\ldots ,T_p,R)$ with respect to the weak convergence of signed measures. More precisely, 

$\mathcal{G}^*(T_1,\ldots ,T_p,R)=\{\mu :$ there exists $\mu _n\in \mathcal{G}(T_1,\ldots ,T_p,R)$ weakly converging to $\mu \}$. 

Finally, the least negative intersection $\Lambda (T_1,\ldots ,T_p,R)$ is the strong submeasure whose action on $\varphi \in C^0(X)$ is given by
\begin{eqnarray*}
\Lambda (T_1,\ldots ,T_p,R)(\varphi )=\sup _{\mu \in \mathcal{G}^*(T_1,\ldots ,T_p,R),~||\mu ||_{neg}=\kappa (T_1,\ldots ,T_p,R)}\mu (\varphi ). 
\end{eqnarray*}
 
Now we state and prove the main results concerning the least negative intersection.   
\begin{theorem}
 Let $X$ be a compact K\"ahler manifold of dimension $k$, $T_1,\ldots ,T_p$ positive closed $(1,1)$ currents on $X$ and $R$ a positive closed $(k-p,k-p)$ current on $X$. Then there is a strong submeasure $\Lambda (T_1,\ldots ,T_p,R)$ with the following properties: 
 
 1) {\bf Symmetry.} $\Lambda (T_1,\ldots ,T_p,R)$ is symmetric in $T_1,\ldots ,T_p$. 
 
 2) {\bf Compatibility with cohomology.} The mass $\Lambda (T_1,\ldots ,T_p, R)(1)$ is the intersection number of cohomology classes $\{T_1\}\cdot \{T_2\}\cdots \{T_p\}\cdot \{R\}$. 
 
 3) {\bf Compatibility with classical wedge intersection.} If $U\subset X$ is an open set on which the wedge intersection $T_1|_U\wedge \ldots \wedge T_p|_U\wedge R|_U$ is classically defined (that is, Bedford-Taylor's monotone convergence is satisfied, see Section \ref{SectionIntersection} for more detail), then 
 \begin{eqnarray*}
 \Lambda (T_1,\ldots ,T_p,R)|_U=T_1|_U\wedge \ldots \wedge T_p|_U\wedge R|_U.
 \end{eqnarray*}
 
 4) For all constants $B$: $$\Lambda (T_1,\ldots ,T_p,R)(\varphi +B)=\Lambda (T_1,\ldots ,T_p,R)(\varphi )+B\{T_1\}\cdots \{T_p\}\cdot \{R\}.$$ 
 
 5) ({\bf Maximalty.}) $\Lambda (T_1,\ldots ,T_p,R)$ is the supremum of $\mathcal{G}^*(T_1,\ldots, T_p,R)$ with respect to the partial order in Definition \ref{DefinitionOrderOnMeasures}. 
 
 6) Let $T_1'$ be another positive closed $(1,1)$ current on $X$. Assume that $\kappa _{T_1,T_2,\ldots ,T_p,R}$, $\kappa _{T_1,T_2,\ldots ,T_p,R}>0$. Then $\Lambda (T_1+T_1',T_2,\ldots ,T_p,R)\geq \Lambda (T_1,T_2,\ldots ,T_p,R)+\Lambda (T_1',T_2,\ldots ,T_p,R)$.
\label{TheoremLeastNegativeIntersection}\end{theorem}

\begin{proof}[Proof of Theorem \ref{TheoremLeastNegativeIntersection}] To show that $\Lambda (T_1,\ldots ,T_p,R)$ is a strong submeasure, by part 3) of Theorem \ref{TheoremSubmeasureBasicProperty}, it suffices to show that $$\sup _{\mu \in \mathcal{G}^*(T_1,\ldots ,T_p,R),~||\mu ||_{neg}=\kappa (T_1,\ldots ,T_p,R)}||\mu ||<\infty .$$
But this is clear, since if we write $\mu =\mu ^+-\mu ^-$ with $||\mu ^{-}||=\mu ^{-}(1)=\kappa (T_1,\ldots ,T_p,R)$, then from $$||\mu ^{+}||=\mu ^{+}(1)=\mu (1)+\mu ^{-}(1)=\{T_1\}\cdot \{T_p\}\cdot \{R\}+\kappa (T_1,\ldots ,T_p,R),$$
we have
$$||\mu ||\leq ||\mu ^+||+||\mu ^-||=\{T_1\}\cdot \{T_p\}\cdot \{R\}+2\kappa (T_1,\ldots ,T_p,R)$$
for all such $\mu $.

1) The symmetry between $T_1,\ldots ,T_p$ is clear since in the definition of $\Lambda (T_1,\ldots ,T_p,R)$ we use all possible good monotone approximations by smooth quasi-psh functions of the quasi-potentials of all $T_1,\ldots ,T_p$. 

2) If $\mu \in \mathcal{G}(T_1,\ldots ,T_p,R)$, then it is a weak convergence limit 
\begin{eqnarray*}
\mu =\lim _{n\rightarrow\infty}(\Omega _1+dd^cu_1^{(n)})\wedge \ldots \wedge (\Omega _p+dd^cu_p^{(n)})\wedge R.
\end{eqnarray*}
Therefore, since $(\Omega _1+dd^cu_1^{(n)})\wedge \ldots \wedge (\Omega _p+dd^cu_p^{(n)})\wedge R(1)$ is the cohomology intersection $\{\Omega _1+dd^cu_1^{(n)}\}\cdots \{\Omega _1+dd^cu_p^{(n)}\}\cdot \{R\} $, which is the same as $\{T_1\}\cdots \{T_p\}\cdot \{R\} $ for all $n$, we obtain the conclusion.  

3) This follows from the definition of classically defined wedge intersections.

4) This follows from the fact that for any $\mu \in \mathcal{G}(T_1,\ldots ,T_p,R)$, a function $\varphi \in C^0(X)$ and a constant $B$, we have
\begin{eqnarray*}
\mu (\varphi +B)=\mu (\varphi )+B\mu (1)=\mu (\varphi )+B\{T_1\}\cdots \{T_p\}\cdot \{R\}. 
\end{eqnarray*} 

5) This is obvious. 

6) This follows readily from the fact that under this assumption, we have $\mathcal{G}^*(T_1,T_2,\ldots ,T_p,R)+\mathcal{G}^*(T_1',T_2,\ldots ,T_p,R)\subset \mathcal{G}^*(T_1+T_1',T_2,\ldots ,T_p,R)$. 
\end{proof}
      
\begin{proof}[Proof of Proposition \ref{PropositionIntersectionProjectiveSpace}]
1) It is known that if $T_i=\Omega _i+dd^cu_i$ is a positive closed $(1,1)$ current in $\mathbb{P}^k$, then there is a good monotone approximation $\{u_i^{(n)}\}_n$ of $u_i$ by smooth quasi-psh functions so that $\Omega _i+dd^cu_i^{(n)}$ is a positive form for all $n$. Therefore, for all $n$, the wedge intersection $(\Omega _i+dd^cu_i^{(n)})\wedge \ldots \wedge (\Omega _p+dd^cu_i^{(n)})\wedge R$ is a positive measure. Any subsequence limit of this sequence will be also a positive measure, and hence the number $\kappa (T_1,\ldots ,T_p,R)$ is $0$. From this and definition, we see that $\Lambda (T_1,\ldots ,T_p,R)$ is in $SM^+(X)$. 

2) Since the wedge intersection $[D]\wedge [D]$ is classically defined {\bf outside of} $D$, and there the resulting is $0$, it follows that $\Lambda ([D],[D])$ has support in $D$. By 1), $\Lambda ([D],[D])$ is in $SM^+(X)$ and its mass is $\Lambda ([D],[D])(1)=\{D\}.\{D\}=1$. Therefore, for all $\varphi \in C^0(X)$, by the positivity of $\Lambda ([D],[D])$ and the fact that it has support in $D$ we have
\begin{eqnarray*}
\Lambda ([D],[D])(\varphi )\leq \Lambda ([D],[D])(\sup _D\varphi ) =\sup _D\varphi . 
\end{eqnarray*}
Now we prove the reverse inequality. Fix a smooth closed $(1,1)$ form $\Omega$ having the same cohomology class as that of $[D]$. From Example 3, we see that $(\Omega +dd^cv)\wedge [D]$ belongs to $\mathcal{G}([D],[D])$ for all smooth function $v$ on $X$. For any point $p\in D$, we choose an other line $D_1\subset X=\mathbb{P}^2$, so that $D\cap D_1=\{p\}$. Then $[D_1]$ has the same cohomology class as $[D]$, and hence we can write $[D_1]=\Omega +dd^cu_1$ for some quasi-psh function $u_1$. By the argument in the proof of 1), we can find a good monotone approximation $\{v_n\}$ of $u_1$ by smooth quasi-psh functions and so that $\Omega +dd^cv_n\geq 0$ for all $n$. Note that the intersection $[D_1]\wedge [D]$ is classically defined, and hence we have
\begin{eqnarray*}
\lim _{n\rightarrow\infty}(\Omega +dd^cv_n)\wedge [D]=\delta _p,
\end{eqnarray*}
the Dirac measure at $p$. Therefore, $\delta _p\in \mathcal{G}^*([D],[D])$ with $||\delta _p||_{neg}=0=\kappa ([D],[D])$. Hence, by definition of $\Lambda ([D],[D])$ we have
\begin{eqnarray*}
\Lambda ([D],[D])(\varphi )\geq \delta _p(\varphi )=\varphi (p),
\end{eqnarray*}
for all $p\in D$. Thus 
\begin{eqnarray*}
\Lambda ([D],[D])(\varphi )\geq \sup _D\varphi,
\end{eqnarray*}
as wanted. 

3) It is well-known that $\{E\}.\{E\}=-1$ in cohomology. We denote by $\pi :X\rightarrow \mathbb{P}^2$ the blowup. Now we fix a point $x_0\in E$ and a line $\Sigma _0$ in $\mathbb{P}^2$ so that the intersection between the strict transform $\widetilde{\Sigma _0}$ and $E$ is $x_0$. If we write $\widetilde{\Sigma _0}=\Omega _0+dd^c v_0$, then there is a good monotone approximation $\{v_n\}$ of $v_0$ so that $\lim _{n\rightarrow\infty}(\Omega _0+dd^cv_n)\wedge [E]=\delta _{x_0}$. Now we use that in $\mathbb{P}^2$ there is a sequence of smooth positive closed $(1,1)$ forms $\omega _n$, in the same cohomology class as $[\Sigma _0]$, so that $\lim _{n\rightarrow \infty}\omega _n=[\Sigma _0]$. Then $\Omega _n=\pi ^*(\omega _n)$ is a sequence of positive closed smooth $(1,1)$ forms on $X$ weakly converging to $[\pi ^*(\Sigma _0)]$. Since in cohomology $\{\pi ^*(\omega _n)\}.\{E\}=0$, it follows that $\pi ^*(\omega _n)\wedge [E]=0$. Since $\pi ^*(\omega _n)-\Omega -dd^c v_n$ is a smooth closed $(1,1)$ form having the same cohomology class as $[E]$, it follows by Example 3 that 
\begin{eqnarray*}
(-\Omega -dd^c v_n)\wedge [E]=(\pi ^*(\omega _n)-\Omega -dd^c v_n)\wedge [E]\in \mathcal{G}([E],[E])
\end{eqnarray*}
for all $n=1,2,\ldots $. Therefore, by taking limit when $n\rightarrow\infty$, we obtain that $-\delta _{x_0}\in \mathcal{G}^*([E],[E])$. Since $-\delta _{x_0}$ is a negative measure and has mass $-1=\{E\}.\{E\}$ negative, it follows that $\kappa _{[E],[E]}=-1$. Therefore, if $\mu\in \mathcal{G}^*([E],[E])$ is so that $||\mu ||_{neg}=\kappa _{[E],[E]}$, it follows that $\mu $ is a negative measure with mass $-1$. Hence for all continuous functions $\varphi$
\begin{eqnarray*}
\Lambda ([E],[E])(\varphi )\geq \sup _{x_0\in E}(-\delta _{x_0}(\varphi ))=  \sup _{x_0\in E}(-\varphi (x_0)).
\end{eqnarray*}
On the other hand, $\Lambda ([E],[E])$ has support in $E$ and so
\begin{eqnarray*}
\Lambda ([E],[E])(\varphi )&=&\sup _{\mu \in  \mathcal{G}^*([E],[E]),||\mu ||_{neg}=-1}\mu (\varphi )\\
&=& \sup _{\mu \in  \mathcal{G}^*([E],[E]),||\mu ||_{neg}=-1}(-\mu ) (-\varphi )\\
&\leq&\sup _{\mu \in  \mathcal{G}^*([E],[E]),||\mu ||_{neg}=-1}\sup _{x_0\in E}(-\varphi (x_0))\\
&=&\sup _{x_0\in E}(-\varphi (x_0)).
\end{eqnarray*}
Here we use that for $\mu$ in the supremum, $-\mu$ is a positive measure of mass $1$. 
\end{proof}      
      
The above arguments and results in Chapter 3 in \cite{demailly1}  yield the following result. 
\begin{theorem}
Let $X$ be a compact K\"ahler manifold, $T_1,\ldots ,T_p$ positive closed $(1,1)$ currents and $R$ a positive closed $(k-p,k-p)$ current. 

1) If $\kappa (T_1,\ldots ,T_p,R)=0$, then the least negative intersection  $\Lambda (T_1,\ldots , T_p,R)$ is in $SM^+(X)$.

2) Let $E_i=\{x\in X:~\nu (T_i,x)>0\}$. Assume that for any $1\leq i_1< i_2< \ldots < i_q\leq p$, the $2p-2q+1$ Hausdorff dimension of $E_{i_1}\cap \ldots \cap E_{i_q}\cap sup (R)$ is $0$. Then $\kappa (T_1,\ldots ,T_p,R)=0$. 

3) Assume that $R=[W]$ is the current of integration on a subvariety $W\subset X$. Assume that for any $1< i_1< i_2< \ldots < i_q\leq p$, the intersection $E_{i_1}\cap \ldots \cap E_{i_q}\cap W$ does not contain any variety of dimension $>p-q$. Then $\kappa (T_1,\ldots ,T_p,R)=0$. 

4) Assume that $\kappa (T_1,\ldots ,T_p,R)=0$. Assume also that there is a Zariski open set $U\subset X$ on which the intersection $T_1|_U\wedge \ldots \wedge T_p|_U\wedge R|_U$ is classically defined, and whose mass is exactly $\{T_1\}\cdots \{T_p\}\cdot \{R\}$. Then $\Lambda (T_1,\ldots, T_p,R)$ is the extension by $0$ of the positive measure $T_1|_U\wedge \ldots \wedge T_p|_U\wedge R|_U$.
\label{TheoremPositiveWedgeIntersection}\end{theorem}
\begin{proof}
We prove for example 3) and 4). 

3) By part 2) of Theorem \ref{TheoremDemaillyRegularisation}, we have good monotone approximations $\{T_i^{(n)}\}$ of $T_i$ by currents with analytic singularities with the following properties: i) $T_i^{(n)}+\epsilon _n\omega \geq 0$ for all $n$ where $\epsilon _n\rightarrow 0$, and ii) $\{x\in X:~\nu (T_i^{(n)})>0\}\subset E_i$ for all $n$ and $i$. Then from  Chapter 3 in \cite{demailly1}, it follows from i) that for each $n$ the intersection $T_1^{(n)}\wedge \ldots \wedge T_p^{(n)}\wedge R$ is classically defined, and it follows from ii) that any cluster point of $T_1^{(n)}\wedge \ldots \wedge T_p^{(n)}\wedge R$ is a positive measure.  By part 2) of Proposition \ref{PropositionPropertyCollectionG}, any such limit measure is in $\mathcal{G}(T_1,\ldots , T_p,R)$. Thus $\kappa (T_1,\ldots ,T_p,R)=0$. 

4) Let $\mu $ be any signed measure in $\mathcal{G}(T_1,\ldots ,T_p,R)$. From the definition, we have that $\mu |_U=T_1|_U\wedge \ldots \wedge T_p|_U\wedge R|_U$ . It follows that the same is true also for $\mu \in \mathcal{G}^*(T_1,\ldots ,T_p,R)$. Since $\kappa (T_1,\ldots ,T_p,R)=0$, it follows that 
\begin{eqnarray*}
\Lambda (T_1,\ldots ,T_p,R)(\varphi )=\sup _{\mu \mbox{ positive measure }\in \mathcal{G}^*(T_1,\ldots ,T_p,R)}\mu (\varphi ). 
\end{eqnarray*}
Hence the claim follows if we can show that $\mathcal{G}^*(T_1,\ldots ,T_p,R)$ has only one measure. In fact, let $\mu \in \mathcal{G}^*(T_1,\ldots ,T_p,R)$ be a positive measure. Then $\mu$ has mass $\{T_1\}\cdots \{T_p\}\cdot \{R\}$, which is the same as that of $T_1|_U\wedge \ldots \wedge T_p|_U\wedge R|_U$, and by the above arguments $\mu |_U=T_1|_U\wedge \ldots \wedge T_p|_U\wedge R|_U$. Therefore, $\mu$ must be the extension by $0$ of $T_1|_U\wedge \ldots \wedge T_p|_U\wedge R|_U$.  
\end{proof} 

\subsection{The range of the Monge-Ampere operator $\Lambda (T,T,\ldots ,T)$}

Let $X$ be a compact K\"ahler manifold of dimension $k$. Here we discuss about the range of the Monge-Ampere operator $T\mapsto \Lambda (T,T,\ldots ,T)$, where $T$ is a positive closed $(1,1)$ current on $X$. We will give some evidence to that the positive strong submeasures in the range, while can be singular, cannot be too singular. 

First, we consider the case where $k=1$. In this case, the range is exactly the set of positive closed $(1,1)$ currents on $X$, and hence is exactly the set of measures on $X$. In particular, there is no positive closed $(1,1)$ current $T$ on $X$ so that $T=\sup _{x\in X}\delta _x$. On the other hand, any proper subvariety of $X$ is a finite number of points, say $Z=\{p_1,\ldots ,p_m\}$, and the positive closed $(1,1)$ current $T=\delta _{p_1}+\ldots +\delta _{p_m}$ satisfies $T\geq \sup _{x\in Z}\delta _x$.

When we move to dimension, we see already in Proposition \ref{PropositionIntersectionProjectiveSpace} that the range of the Monge-Ampere operators can be much more than just measures. In fact, the same proof as that of part 2 of  Proposition \ref{PropositionIntersectionProjectiveSpace} shows the following: 
\begin{lemma}
Let $D$ be an irreducible curve in $\mathbb{P}^2$, with degree $deg(D)$. Then $$\deg (D)\sup _{x\in D}\delta _x\leq \Lambda ([D],[D])\geq \deg (D)^2\sup _{x\in D}\delta _x.$$ 
\label{LemmaSelfIntersectionCurves}\end{lemma}

Now let $T$ be any positive closed $(1,1)$ current on $\mathbb{P}^2$. By Siu's decomposition theorem, we can write it uniquely as
\begin{eqnarray*}
T=T_0+\sum _na_n[D_n],
\end{eqnarray*}
where $T_0$ is a positive closed $(1,1)$ current with no mass on any curve, and $D_n$ are curves, so that $\sum _na_n \deg (D_n)<\infty$. The proof of Proposition \ref{PropositionIntersectionProjectiveSpace} shows that for any two positive closed $(1,1)$ currents $T_1,T_2$ on $\mathbb{P}^2$ we have $\kappa _{T_1,T_2}\geq 0$. Therefore, by Theorem \ref{TheoremLeastNegativeIntersectionShortVersion}, we have
\begin{eqnarray*}
\Lambda (T,T)&\geq& \Lambda (T_0,T_0)+\sum _{n}a_n^2\Lambda ([D_n],[D_n])+  2\sum _{n}a_n \Lambda (T_0,[D_n])+\sum _{n\not =m}a_na_m\Lambda ([D_n],[D_m])\\
&\geq&\sum _{n}a_n^2\deg (D_n)\sup _{x\in D_n}\delta _x.
\end{eqnarray*}
If we choose the curves $\{D_n\}$ so that $\cup _nD_n$ is dense in $\mathbb{P}^2$, we see that the Monge-Ampere $\Lambda (T,T)$ can be very singular. However, the above estimate cannot imply that $\Lambda (T,T)\geq \sup _{x\in A}\delta _x$ for some subset $A$ not contained in any proper subvariety of $\mathbb{P}^2$.  In connection to the case of dimension $1$ above, we conjecture that in fact this speculation is true. 

We can work similarly on higher dimensional manifolds $X$. This leads us to state the following conjecture. The analysis in the case of dimension 1, and also that on $\mathbb{P}^2$ and $\mathbb{P}^k$, shows that this conjecture is sharp. 

{\bf Conjecture A.} Let $X$ be a compact K\"ahler manifold of dimension $k$. If there is a positive closed $(1,1)$ current $T$ on $X$ and a subset $A\subset X$ so that $\Lambda (T_1=T,T_2=T,\ldots ,T_k=T)\geq \sup _{x\in A}\delta _x$, then there is a proper subvariety $Z\subset X$ containing $A$. In particular, there is no positive closed $(1,1)$ current $T$ on $X$  so that $\Lambda (T_1=T,T_2=T,\ldots ,T_k=T)\geq \sup _{x\in X}\delta _x$. 

We also state the following conjecture. Conjecture B is trivial in the case $\mu _1$ is a positive measure, since in that case we have $\mu _1=\mu _2$.  In particular, it is true in dimension $1$. The analysis in the case of dimension $1$ shows that we cannot get rid of the assumption that $||\mu _1||=||\mu _2||$ in the Conjecture B. 

{\bf Conjecture B.} Let $X$ be a compact K\"ahler manifold, and let $\mu _1,\mu _2$ be two positive strong submeasures on $X$ so that $||\mu _1||=||\mu _2||$ and $\mu _1\geq \mu _2$. Assume that there is a positive closed $(1,1)$ current $T_1$ on $X$ so that $\Lambda (T_1,T_1,\ldots ,T_1)=\mu _1$. Then there is a positive closed $(1,1)$ current $T_2$ on $X$ so that $\Lambda (T_2,T_2,\ldots ,T_2)=\mu _2$.  

A corollary of Conjecture B is the following. Let $\mu$ be a positive strong submeasure on $X$ so that there exists a positive closed $(1,1)$ current $T$ for which $\Lambda (T,T,\ldots ,T)=\mu$. Let $\nu$ be a positive measure so that $\mu \geq \nu $ and $|||\nu ||=||\mu ||$. Then there exists a positive closed $(1,1)$ current $T'$ so that $\Lambda (T',T',\ldots ,T')=\nu$. 

\section{Invariant positive strong submeasures and Entropy}\label{SectionVariationalPrinciple}

In this section we explore invariant positive strong submeasures of meromorphic maps and their entropies. We start with the case of continuous maps of compact Hausdorff spaces to provide some intuitions and ideas. 

\subsection{The case of continuous maps on compact Hausdorff spaces}
 
 We first recall relevant definitions about entropy of an invariant measure and the Variational Principle, see \cite{goodwyn, goodman}. Let $X$ be a compact Hausdorff space and $f:X\rightarrow X$ a continuous map. Let $\mu$ be a probability Borel-measure on $X$ which is invariant by $f$, that is $f_*(\mu )=\mu$. We next define the entropy $h_{\mu}(f)$ of $f$ with respect to $\mu$. We say that a finite collection of Borel sets $\alpha $ is a $\mu$-partition if $\mu (\bigcup _{A\in \alpha }A)=1$ and $\mu (A\cap B)=0$ whenever $A,B\in \alpha$ and $A\not= B$. Given a $\mu$-partition $\alpha$, we define
 \begin{eqnarray*}
 H_{\mu}(\alpha )=-\sum _{A\in \alpha}\mu (A)\log \mu (A). 
 \end{eqnarray*}
 If $\alpha$ and $\beta $ are $\mu$-partitions, then $\alpha V \beta :=\{A\cap B:$ $A\in \alpha$ and $B\in \beta \}$ is also a $\mu$-partition. Similarly, for all $n$ the collection $V_{i=0}^{n-1}f^{-(i)}(\alpha )$ $:=\{A_0\cap A_1\cap \ldots \cap A_{n-1}:$ $A_0\in \alpha ,$ $A_1\in f^{-1}(\alpha ),$ $\ldots ,$ $A_{n-1}\in f^{-(n-1)}(\alpha )\}$ is also a $\mu $-partition. We define: 
 \begin{equation}
 h_{\mu }(f,\alpha )=\lim _{n\rightarrow\infty}\frac{1}{n}H_{\mu}(V_{i=0}^{n-1}f^{-(i)}(\alpha ))=\inf _n\frac{1}{n}H_{\mu}(V_{i=0}^{n-1}f^{-(i)}(\alpha )),
 \label{EquationMeasureEntropy}\end{equation} 
(the above limit always exists) and 

$h_{\mu}(f):=$ $\sup \{h_{\mu (\alpha )}:$ $\alpha$ runs all over $\mu$-partitions$\}$. 

Recall (see \cite{rudin}) that a Borel-measure $\mu$ of finite mass is regular if for every Borel set $E$: i) $\mu (E)=\inf \{\mu (V):$ $V$ open, $E\subset V\}$, and ii) $\mu (E)=\sup \{\mu (K):$ $K$ compact, $K\subset E\}$. Note that if $X$ is a compact metric space, then any Borel measure of finite mass is regular.   

The Variational Principle, an important result on dynamics of continuous maps, is as follows \cite{goodwyn, goodman}.
\begin{theorem}
Let $X$ be a compact Hausdorff space and $f:X\rightarrow X$ a continuous map. Then 

$h_{top}(f)=$ $\sup \{h_{\mu }(f):$ $\mu$ runs all over probability regular Borel measures $\mu$ invariant by $f\}$. 
\label{TheoremVariationalPrincipleMeasure}\end{theorem}
It is known, however, that the supremum in the theorem may not be attained by any such invariant measure $\mu$, even if $X$ is a compact metric space. In contrast, here, we show that with an appropriate definition of entropy, positive strong submeasures also fit naturally with the Variational Principle and provide the desired maximum. 

First, we note that even though we defined in previous sections strong submeasures only for compact metric spaces, this definition extends easily to the case of compact Hausdorff spaces. We still have the Hahn-Banach theorem on compact Hausdorff spaces, and hence we can define a strong submeasure $\mu$ by one of the following two equivalent definitions: i) $\mu$ is a bounded and sublinear operator on $C^0(X)$, or ii) $\mu =\sup _{\nu \in \mathcal{G}}\nu$, where $\mathcal{G}$ is a non-empty collection of signed regular Borel-measures on $X$ whose norms are uniformly bounded from above. Such a strong submeasure $\mu$ is positive if moreover $\mathcal{G}$ in ii) can be chosen to consist of only positive measures. 

As in the previous sections, we can then define the pushforward of $\mu$ by a continuous map $f:X\rightarrow X$. We have the following property, which is stronger than 7) of Theorem \ref{TheoremSubmeasurePushforwardMeromorphic}. The proof of the result is similar to, and simpler than that of Theorem \ref{TheoremSubmeasurePushforwardMeromorphic}, and hence is omitted. 
\begin{lemma}
Let $X$ be a compact Hausdorff space and $f:X\rightarrow X$ a continuous map. Let $\mu$ be a positive strong submeasure on $X$, and assume that $\mathcal{G}$ is any non-empty collection of positive measures on $X$ such that  $\mu =\sup _{\nu \in \mathcal{G}}\nu$. Then
\begin{eqnarray*}
f_*(\mu )=\sup _{\nu \in \mathcal{G}}f_*(\nu ). 
\end{eqnarray*}
\label{LemmaPushforwardStrongSubmeasure}\end{lemma}
 
Now we define an appropriate notion of entropy for a positive strong submeasure $\mu$ which is invariant by $f$. A first try would be to naively  adapt (\ref{EquationMeasureEntropy}) to the more general case of positive strong submeasures, and then proceed as before. However, this is not appropriate, as the readers can readily check with the simplest case of identity maps on spaces with infinitely many points. In this case, there are many positive strong submeasures with mass $1$ and invariant by $f$, whose entropy, according to the above definition, can be as large as desired and even can be infinity. On the other hand, recall that the topological entropy of the identity map is $0$. (A refinement of this naive version turns out to be appropriate, see the remarks at the end of this section.)  

We instead proceed as follows. Given $\mu$ a positive strong submeasure invariant by $f$ and any regular measure $\nu \leq \mu$, there is a regular measure $\nu '$ so that $\nu '$ is invariant by $f$, $\nu '\leq \mu$ and $\nu '$ has the same mass as $\nu$. Such a measure $\nu '$ can be constructed as any cluster point of the sequence 
\begin{eqnarray*}
\frac{1}{n}\sum _{i=0}^{n-1}f^i_*(\nu ). 
\end{eqnarray*}
We define $\mathcal{G}(f,\mu )=\{\nu :$ $\nu $ is a regular Borel-measure invariant by $f$, and $\nu\leq \mu\}$. Finally, we define the desired entropy as follows:
\begin{equation}
h_{\mu}(f):=\sup _{\nu \in \mathcal{G}(f,\mu )}h_{\nu}(f).
\label{EquationEntropySubmeasure}\end{equation}

We now can prove the Variational Principle for positive strong submeasures. We first show that if $\mu$ is any positive strong submeasure of mass $1$ and invariant by $f$, then $h_{\mu }(f)\leq h_{top}(f)$. To this end, we need only to observe that for any $\nu \in \mathcal{G}(f,\mu )$, then the mass of $\nu $ is $\leq 1$, and hence $h_{\nu }(f)\leq h_{top}(f)$ by Theorem \ref{TheoremVariationalPrincipleMeasure}. 

We finish the proof by showing that there is a positive strong submeasure $\mu$ of mass $\leq 1$ and invariant by $f$ so that $h_{\mu }(f)=h_{top}(f)$. To this end, we let $\mathcal{G}=\{\nu :$ $\nu$ is a regular measure of mass $\leq 1$ and is invariant by $f\}$. This set is non-empty (it contains at least the measure zero.)  We define $\mu =\sup _{\nu \in \mathcal{G}}\nu$. By Lemma \ref{LemmaPushforwardStrongSubmeasure}, we have that 
\begin{eqnarray*}
f_*(\mu )=\sup _{\nu \in \mathcal{G}}f_*(\nu )=\sup _{\nu}\nu =\mu.
\end{eqnarray*}
Hence, $\mu$ is invariant by $f$. Moreover, from the definition of $\mathcal{G}$, it follows that $\mu$ has mass $\leq 1$. Finally, by (\ref{EquationEntropySubmeasure}) and Theorem \ref{TheoremVariationalPrincipleMeasure} we have that $h_{\mu }(f)=h_{top}(f)$. 
 
{\bf Remarks.} Besides fitting naturally with the Variational Principle, definition (\ref{EquationEntropySubmeasure}) is also compatible with the philosophy that a property of a positive strong submeasure $\mu$ should be related to the supremum of the same property of measures $\nu \leq \mu$. We have seen some instances of this philosophy above: the definition of $\mu$ itself is such a supremum, and Lemma \ref{LemmaPushforwardStrongSubmeasure}. Since entropy is related to invariant measures of $f$, in definition (\ref{EquationEntropySubmeasure}) we have restricted to only $\mathcal{G}(f,\mu )$.

While we always have that for a positive strong submeasure $\mu$ invariant by a continuous map $f$ then $\mu \geq \sup _{\nu \in \mathcal{G}(f,\mu)}\nu$, it is not always true that $\mu =\sup _{\nu \in \mathcal{G}(f,\mu )}\nu$. In fact, assume that $f$ is not the identity map and $X$ has at least $2$ elements. Let $\mu =\sup _{x\in X}\delta _x$. Then it is easy to check that $\mu$ has mass $1$ and is $f$-invariant. But it is not true that $\mu =\sup _{\nu \in \mathcal{G}(f,\mu )}\nu$. In fact, assume otherwise that $\mu =\sup _{\nu \in \mathcal{G}(f,\mu )}\nu$. Let $x_0\in X$ be such that $f(x_0)\not= x_0$.  Then there is a sequence $\nu _n\in \mathcal{G}(f,\mu )$ so that $\nu _n(x_0)\rightarrow 1$. Hence, since $\nu _n$ all has mass $\leq 1$, it follows that $\nu _n$ converges to $\delta _{x_0}$. Since $f_*(\nu _n)=\nu _n$ for all $n$, it follows also that $f_*(\delta _{x_0})=\delta _0$. This, in turn, means that $f(x_0)=x_0$, which is  a contradiction.  

While the naive version discussed in the paragraph after Lemma \ref{LemmaPushforwardStrongSubmeasure} of the entropy $h_{\mu}(f,\alpha )$ is not a good one, here we present a refinement which is compatible to  (\ref{EquationEntropySubmeasure}) and hence the Variational Principle. Define $\mathcal{P}$ the set of all finite collections $\alpha$ of Borel sets of $X$. For $\nu \in \mathcal{G}(f,\mu )$, if $\alpha$ is a $\nu$-partition, then we define $\widetilde{h}_{\nu }(f,\alpha )=h_{\nu}(f,\alpha )$ as in (\ref{EquationMeasureEntropy}), otherwise we define $\widetilde{h}_{\nu }(f,\alpha )=0$. By (\ref{EquationEntropySubmeasure}), we have
\begin{eqnarray*}
h_{\mu}(f)&=&\sup _{\nu \in \mathcal{G}(f,\mu )}h_{\nu }(f)\\
&=&\sup _{\nu \in \mathcal{G}(f,\mu )}\sup _{\alpha \in \mathcal{P}}\widetilde{h}_{\nu }(f,\alpha )\\
&=&\sup _{\alpha \in \mathcal{P}}[\sup _{\nu \in \mathcal{G}(f,\mu )}\widetilde{h}_{\nu }(f,\alpha )].
\end{eqnarray*}
This suggests us, in parallel with (\ref{EquationMeasureEntropy}), to define for any finite collection of Borel sets $\alpha$, the quantity 
\begin{eqnarray*}
h_{\mu}(f,\alpha )=\sup _{\nu \in \mathcal{G}(f,\mu )}\widetilde{h}_{\nu }(f,\alpha ).
\end{eqnarray*}
Then, we have, as in the classical case: $h_{\mu}(f)=\sup _{\alpha \in \mathcal{P}}h_{\mu}(f,\alpha )$. 

\subsection{The case of meromorphic maps on compact complex varieties}
We now prove results concerning invariant positive strong submeasures of meromorphic maps $f:X\dashrightarrow X$  of compact complex varieties and their entropies. 

\begin{proof}[Proof of Proposition \ref{PropositionKeyEntropy}] By part ii of Theorem \ref{TheoremSubmeasurePushforwardMeromorphicShortVersion} and the fact that $\phi _f$ and $\pi _1$ are continuous maps, we need to consider only the case where $\hat{\mu}$ is a positive measure. We need to show that for all $\varphi \in C^0(X)$ 
\begin{eqnarray*}
f_* (\pi _1)_*(\hat{\mu })(\varphi )\geq (\pi _1)_*(\phi _f)_*(\hat{\mu} )(\varphi ). 
\end{eqnarray*}
By definition 
\begin{eqnarray*}
f_* (\pi _1)_*(\hat{\mu} )(\varphi )=\inf _{\psi \in C^0(X,\geq f^*(\varphi ))}(\pi _1)_*(\hat{\mu} )(\psi ) = \inf _{\psi \in C^0(X,\geq f^*(\varphi ))} \hat{\mu} (\pi _1^*(\psi )). 
\end{eqnarray*}
Let $U=X\backslash I(f)$. Then $U$ is an open dense subset of $X$ and $\pi _1^{-1}(U)$ is an open dense subset of $\Gamma _{f,\infty}$. Note that $\pi _1^*\circ f^*(\varphi )=\phi _f^*\pi _1^*(\varphi )$ on $\pi _1^{-1}(U)$. Moreover, note that the function $\phi _f^*\pi _1^*(\varphi )$ is continuous on the whole $\Gamma _{f,\infty}$. Therefore, for every $\psi \in C^0(X,\geq f^*(\varphi ))$, we have that $\pi _1^*(\psi )\geq \phi _f^*\pi _1^*(\varphi )$. This implies that
\begin{eqnarray*}
f_* (\pi _1)_*(\hat{\mu} )(\varphi )\geq  \hat{\mu} (\phi _f^*\pi _1^*(\varphi ))=(\pi _1)_*(\phi _f)_*(\hat{\mu} )(\varphi ). 
\end{eqnarray*}

If we choose an example $f:X\dashrightarrow X$ having a point $x\in I(f)$ such that $f_*(\delta _x)$ is not a measure, then for any $\hat{x}\in \Gamma _{f,\infty}$  so that $\pi _1(\hat{x})=x$ and $\hat{\mu}=\delta _{\hat{x}}$, we have $f_* (\pi _1)_*(\hat{\mu })\not= (\pi _1)_*(\phi _f)_*(\hat{\mu} )$, since the LHS is not a measure while the RHS is a measure. 
\end{proof}

\begin{theorem} Let $f:X\dashrightarrow X$ be a dominant meromorphic map of a compact complex variety. Let $0\not= \mu _0\in SM^+(X)$.

1) Let $\hat{\mu}$ be a $\phi _f$-invariant positive strong submeasure. Then $\mu =(\pi _1)_*(\hat{\mu})$ satisfies: $f_*(\mu )\geq \mu$.
 
2) Let $\mu _0$ be a positive strong submeasure on $X$, and $\mu$ is a cluster point of Cesaro's averages $\frac{1}{n}\sum _{j=0}^n(f_*)^j(\mu _0)$. Then $f_*(\mu )\geq \mu$.

3) If $f_*(\mu _0)=\mu _0$, then there exists a non-zero measure $\hat{\mu _0}$ on $\Gamma _{f,\infty}$ so that $(\phi _f)_*(\hat{\mu _0})=\hat{\mu _0}$ and $(\pi _1)_*(\hat{\mu _0})\leq \mu _0$. Moreover, the set $\{\hat{\mu} \in SM^+(\Gamma _{f,\infty}):~(\pi _1)_*(\hat{\mu})\leq \mu ,~(\phi _f)_*(\hat{\mu} )= \hat{\mu}\}$ has a maximum, denoted by $Inv(\pi _1, \mu)$. 

4) If $f_*(\mu _0)\leq \mu _0$, then the set  $\{\mu \in SM^+(X):~ \mu \leq \mu _0,~f_*(\mu )= \mu \}$ is non-empty and has a largest element, denoted by $Inv(\leq \mu _0)$. Moreover, $Inv(\leq \mu _0)=\lim _{n\rightarrow\infty}(f_*)^n(\mu _0)$. 

5) If $f_*(\mu _0)\geq \mu _0$, then the set $\{\mu \in SM^+(X):~ \mu \geq \mu _0,~f_*(\mu )= \mu \}$ is non-empty and has a smallest element, denoted by $Inv(\geq \mu _0)$. 
\label{TheoremInvariantMeasures}\end{theorem}
\begin{proof}
1) This follows immediately from Proposition \ref{PropositionKeyEntropy}. 

2) We note that for $\mu _n=\frac{1}{n}\sum _{j=0}^n(f_*)^j(\mu _0)$, by part 7 of Theorem \ref{TheoremSubmeasurePushforwardMeromorphic} we have
\begin{eqnarray*}
f_*(\mu _n)-\mu _n&=&f_*(\frac{1}{n}\sum _{j=0}^n(f_*)^j(\mu _0)) -  \frac{1}{n}\sum _{j=0}^n(f_*)^j(\mu _0)\\
&\geq& \frac{1}{n}\sum _{j=0}^n(f_*)^{j+1}(\mu _0)-\frac{1}{n}\sum _{j=0}^n(f_*)^j(\mu _0)\\
&=&\frac{1}{n}(f_*)^{n+1}(\mu _0)-\frac{1}{n}\mu _0,
\end{eqnarray*}
and the latter converges to $0$ in $SM(X)$. Therefore, if $\mu =\lim _{j\rightarrow\infty}\mu _{n_j}$, then any cluster point of $f_*(\mu _{n,j})$ is $\geq \mu$.  Hence, by part 4 of Theorem \ref{TheoremSubmeasurePushforwardMeromorphic} we have $f_*(\mu)\geq \mu$. 

3) Choose any non-zero measure $\nu _0$ on $\Gamma _{f,\infty }$ so that $(\pi _1)_*(\nu _0)\leq \mu$. Then  by Proposition  \ref{PropositionKeyEntropy} we have $$(\pi _1)_*(\phi _f)_*(\nu _0)\leq f_*(\pi _1)_*(\nu _0)\leq f_*(\mu) =\mu.$$
Hence any cluster point $\nu$ of the Cesaro's averages $\frac{1}{n}\sum _{j=0}^n(\phi _f)_*(\nu _0)$ will satisfy $(\pi _1)_*(\nu )\leq \mu$. Since $\phi _f$ is continuous, it follows that $(\phi _f)_*(\nu )=\nu$. If we define $Inv(\pi _1,\mu)=\sup \{\hat{\mu} \in SM^+(\Gamma _{f,\infty}):~(\pi _1)_*(\hat{\mu})\leq \mu ,~(\phi _f)_*(\hat{\mu} )= \hat{\mu}\}$, then it is also an element of $SM^+(\Gamma _{f,\infty})$. Moreover, since $\phi _f$ is continuous, we can check easily that $Inv(\pi _1,\mu)$  is $\phi _f$-invariant. 

4) Since $\mu _n=(f_*)^n(\mu _0)$ is a decreasing sequence, it has a limit which we denote by $Inv(\geq \mu _0)$, which is an element of $SM^+(X)$. Moreover, the sequence $f_*(\mu _n)=\mu _{n+1}$  also converges to $Inv(\geq \mu _0)$. We then have that $f_*(Inv(\geq \mu _0))\geq Inv(\geq \mu _0)$ by part 4 of Theorem \ref{TheoremSubmeasurePushforwardMeromorphic}. On the other hand, since $Inv (\geq \mu _0)\leq \mu _n$ for all $n$, it follows that $f_*Inv (\geq \mu _0)\leq f_*(\mu _n)$ for all $n$, and hence $f_*Inv (\geq \mu _0)\leq \lim _n f_*(\mu _n)=Inv (\geq \mu _0)$. Combining all inequalities we obtain $f_*(Inv (\geq \mu _0))=Inv (\geq \mu _0)$. 

To finish the proof, we will show that if $\mu \leq \mu _0$ and $f_*(\mu )=\mu$, then $\mu \leq Inv (\geq \mu _0)$. In fact, under the assumptions about $\mu$ we have
\begin{eqnarray*}
\mu =(f_*)^n(\mu) \leq (f_*)^n(\mu _0)
\end{eqnarray*}
for all positive integers $n$. Hence, by taking the limit we obtain that $\mu \leq Inv (\geq \mu _0)$.

 5) We can assume that the mass of $\mu _0$ is $1$. The set $\mathcal{G}:=\{\mu \in SM^+(X):~ \mu \geq \mu _0,~f_*(\mu )= \mu \}$ is non-empty because $\sup _{x\in X}\delta _x$ is one of its elements.

Now let $\mathcal{H}=\{\nu\in M^+(X):~\nu \leq \mu, ~\forall \mu \in \mathcal{G}\}$. Note that $\mathcal{H}$  is non-empty because if $\nu \leq \mu _0$, then $\nu \in \mathcal{H}$. We define $Inv(\geq \mu _0)=\sup _{\nu \in \mathcal{H}}\nu$. Then $Inv(\geq \mu _0)$ is in $SM^+(X)$, and it is the largest element of $SM^+(X)$ which is smaller than $\mu$ for all $\mu \in \mathcal{G}$. Moreover, by construction we see easily that $Inv(\geq \mu _0)\geq \mu _0$. 

We now finish the proof by showing that $f_*(Inv(\geq \mu _0))=Inv(\geq \mu _0)$. First, we show that $Inv(\geq \mu _0)\geq f_*(Inv(\geq \mu _0))$. In fact, it is easy to check that if $\nu \in M^+(X)$ is so that $\nu \leq Inv(\geq \mu _0)$, then $\nu \in \mathcal{H}$. Hence, by part 6 of Theorem \ref{TheoremSubmeasurePushforwardMeromorphic} we have that for all $\mu \in \mathcal{G}$
\begin{eqnarray*}
f_*(Inv(\geq \mu _0))=\sup _{\nu \in \mathcal{H}}f_*(\nu )\leq f_*(\mu)=\mu . 
\end{eqnarray*}
Therefore, by the definition of $Inv(\geq \mu _0)$, we get that $f_*(Inv(\geq \mu _0))\leq Inv(\geq \mu _0)$. Therefore, by part 4 above, we have that if $\mu _{\infty}$ is the limit point of $\mu _n=(f_*)^n(Inv(\geq \mu _0))$, then $f_*(\mu _{\infty})=\mu _{\infty}$. Moreover $Inv(\geq \mu _0)\geq \mu _{\infty}$.

On the other hand, from the fact that $Inv(\geq \mu _0)\geq \mu _0$ and $f_*(\mu _0)\geq \mu _0$ we have
\begin{eqnarray*}
\mu _n =(f_*)^n(Inv(\geq \mu _0))\geq (f_*)^n(\mu _0)\geq \mu _0,
\end{eqnarray*}
for all $n$. Taking the limit we obtain $\mu _{\infty}\geq \mu _0$, that is $\mu _{\infty}\in \mathcal{G}$. Hence, by definition $\mu _{\infty}\geq Inv(\geq \mu _0)$. 

From the above two inequalities, we deduce that $ Inv(\geq \mu _0) =\mu _{\infty}$, which implies that $ Inv(\geq \mu _0)$ is the smallest element of $\mathcal{G}$.
\end{proof}

\begin{proof}[Proof of Proposition \ref{PropositionMotivationMeasureEntropy}] Let $\hat{\mu}\in SM^+(\Gamma _{f,\infty})$ such that $(\pi _1)_*(\hat{\mu})={\mu}$.  Since $\mu $ has no mass on $I_{\infty}(f)$, it follows that $\hat{\mu}$ has no mass on $\pi _1^{-1}(I_{\infty}(f))$. Therefore, the mass of $\hat{\mu}$ is concentrated on the set $\{(x,f(x),f^2(x),\ldots ):~x\in X\backslash I_{\infty}(f)\}$. Since $\pi _1$ is an isomorphism on $X\backslash I_{\infty}(f)$, it follows that $\hat{\mu}$ is $\phi _f$-invariant and 
\begin{eqnarray*}
h_{\mu}(f)=h_{\hat{\mu}}(\phi _f). 
\end{eqnarray*} 

Now if $\nu$ is $\phi _f$-invariant and $(\pi _1)_*(\nu )\leq \mu$, then $(\pi _1)_*(\nu)$ has no mass on $I_{\infty}(f)$ and hence the same argument as above shows that $(\pi _1)_*(\nu)$ is $f$-invariant and
\begin{eqnarray*}
h_{\nu}(\phi _f)=h_{(\pi _1)_*(\nu)}(f)\leq h_{\mu}(f).
\end{eqnarray*}
\end{proof}

Now we conclude this subsection with some remarks on the measure entropy $h_{\mu}(f)$, for $f$-invariant positive strong submeasures $\mu$. First, by the Variational Principle for the continuous map $\phi _f$ and the definition of $h_{\mu}(f)$, we have that $h_{\mu}(f)\leq h_{top}(f)$, and the latter is bounded from above by dynamical degrees of $f$ by \cite{dinh-sibony3}. Also, by the properties proven above, it is not hard to see that $\max _{\mu}h_{\mu}(f)=h_{top}(f)$, where $\mu$ runs all over $f$-invariant positive strong submeasures of mass $1$.   

\section{Applications to transcendental maps on $\mathbb{C}$ and $\mathbb{C}^2$}

While ergodic properties (entropy, invariant measures) of meromorphic maps on $\mathbb{C}$ and $\mathbb{C}^2$ are intensively studied, it seems that these properties are very rarely discussed or explicitly computed for transcendental maps (such as $f(x)=ae^x+b$) in the literature. (Other properties, though, such as Fatou and Julia sets of transcendental holomorphic maps on $\mathbb{C}$, are extensively studied in the literature by many researchers, starting from the work Fatou \cite{fatou}.) In this section, we show that the same ideas used in the previous sections can be applied to this. 

\subsection{The case of transcendental maps on $\mathbb{C}$}

Consider the case of a dominant holomorphic map $f:\mathbb{C}\rightarrow \mathbb{C}$. We let $X=\mathbb{P}^1$. While $f$ cannot be extended to a meromorphic map on $X$ if $f$ is transcendental, the set $U=\mathbb{C}$ where $f$ is defined is a dense open set of $X$, and its image $f(U)$ contains an open dense subset of $X$. Therefore, we can use the results in Sections 2 and 4 to define the pushforward of positive strong submeasures on $X$. (Remark: In contrast, it is in general unknown how to pullback positive strong submeasures on $X$, since $f$ does not have finite fibres.)

 We can then define $\Gamma _{f,\infty}\subset X^{\mathbb{N}}$ and an associated continuous map $\phi _f:\Gamma _{f,\infty}\rightarrow \Gamma _{f,\infty}$ as before, and define the topological entropy by the formula
 \begin{eqnarray*}
 h_{top}(f):=h_{top}(\phi _f).
 \end{eqnarray*} 
We conjecture that if $f$ is a generic transcendental map then $h_{top}(f)$ is infinity. This is to conform with the fact that the dynamical degree of $f$, if it can be defined in this case, should be $+\infty$. 

Now we discuss in more detail the pushforward of a positive measure on $X$ by a transcendental map $f:X\rightarrow X$. Let $x_0$ be the point at infinity of $X=\mathbb{P}^1$. Let $\mu$ be a probability measure on $X$. Then we can write $\mu =\mu _0+a \delta _{x_0}$, where $0\leq a\leq 1$. Then $f_*(\mu _0)$ is well-defined as a measure. Since $f$ is transcendental, it follows by Picard's theorem $f(V\cap \mathbb{C})$ is dense in $X$ for all open neighbourhood $V$ of $x_0$. From this we obtain that $f_*(\delta _{x_0})=\mu _X:=\sup _{x\in X}\delta _x$. We obtain the formula: 
\begin{eqnarray*}
f_*(\mu )=f_*(\mu _0)+a \mu _X. 
\end{eqnarray*}  
We note that $f_*(\mu _X)=\mu _X$. 

It follows that if $\mu _{\infty}$ is a cluster point of the Cesaro's averages $\frac{1}{n}\sum _{j=0}^n(f_*)^j(\mu)$ converges, where $\mu$ is a probability measure, then we can write $\mu _{\infty}=\widetilde{\mu _{\infty}}+b\delta _{x_0} +a \mu _X$.  Here $\widetilde{\mu _{\infty}}$ has no mass on $x_0$, $a,b\geq 0$, $a+b\leq 1$. It can be checked easily that $\widetilde{\mu _{\infty}}$ is $f$-invariant. Hence $f_*(\mu _{\infty})=\widetilde{\mu _{\infty}}+(a+b)\mu _X$ is $f$-invariant.  

By Picard's theorem for essential singularities of one complex variable, if $\nu$ is a measure with no mass on $x_0$ and is $f$-invariant, then the support of $\nu$ has at most $2$ points. Therefore $h_{\nu}(f)=0$.  Hence, $f$ has no interesting invariant measure with no mass on $x_0$.  The above calculation then shows that the only interesting invariant positive strong submeasure for $f$ is $\mu _X$. 

\subsection{The case of transcendental maps on $\mathbb{C}^2$}

Now, we can apply the above arguments for the case of a transcendental map $f:\mathbb{C}^2\rightarrow \mathbb{C}^2$. The only caveat is that now we have not only one, but many, compact K\"aler surfaces $X$ containing $\mathbb{C}^2$ as an open dense set. We denote by $\mathcal{K}(\mathbb{C}^2)$ the set of all such compact K\"ahler surfaces $X$. For each $X\in \mathcal{K}(\mathbb{C}^2)$, we denote by $f_X$ the associated one on $X$. Then following the idea in \cite{truong3}, we will define a notion of K\"ahler topological entropy for $f$ as follows: 
\begin{eqnarray*}
h_{top, \mathcal{K}}(f):=\inf _{X\in \mathcal{K}(\mathbb{C}^2)}h_{top}(f_X). 
\end{eqnarray*}
Here $h_{top}(f_X)$ is, as before, defined using the infinity graph $\Gamma _{f_X,\infty}$, which is the closure of $\{(x,f(x),f^(x),\ldots ):~x\in \mathbb{C}^2\}$ in $X^{\mathbb{N}}$, and the shift map $\phi _{f_X}(x_1,x_2,\ldots )=(x_2, x_3,\ldots )$.

We again conjecture that if $f$ is a generic transcendental map then $h_{top, \mathcal{K}}(f)=\infty$. For a generic transcendental map, probably we can show that $f(U)$ is dense in $\mathbb{C}^2$ for all open neighborhood of a point at infinity. In this case, for any $X\in \mathcal{K}(\mathbb{C}^2)$ and  a positive measure $\mu$ on $X$, which is decomposed as $\mu =\mu _0+\mu _1$, where $\mu _0$ has mass only in $\mathbb{C}^2$ and $\mu _1$ has only mass in $X\backslash \mathbb{C}^2$: 
\begin{eqnarray*}
f_*(\mu )=f_*(\mu _0)+a \mu _X,
\end{eqnarray*} 
where $\mu _X=\sup _{x\in X}\delta _x$. Thus we have a similar behaviour of the pushforward of positive measures across all the compactifications $X$. This is in contrast to the case of meromorphic maps in $\mathbb{C}^2$ as we mentioned previously.

\end{document}